\documentclass[12pt]{article}

\usepackage[english]{babel}
\usepackage{amsmath,amsthm}
\usepackage{amsfonts}
\usepackage{graphicx}
\usepackage{epsfig}
\usepackage{epstopdf}
\usepackage{url}
\usepackage{amssymb}
\usepackage{color}
\usepackage[%
breaklinks%
,colorlinks%
,linkcolor=red
,anchorcolor=blue%
,pagecolor=blue%
,citecolor=blue%
,bookmarks=false%
]{hyperref}

\theoremstyle{Definition}

\numberwithin{equation}{section}

\newtheorem{theorem}{Theorem}
\newtheorem{corollary}{Corollary}
\newtheorem{lemma}{Lemma}
\newtheorem{remark}{Remark}
\newtheorem{definition}{Definition}
\newtheorem{example}{Example}
\newtheorem{conditions}{Conditions}

\def\S{{\mathcal{S}}}

\def\para{\vspace{4 mm}}

\def\cL{\mathcal{L}}
\def\cV{\mathcal{V}}
\def\cW{\mathcal{W}}

\def\cM{\mathcal{M}}
\def\cU{\mathcal{U}}
\def\K{\mathbb{K}}

\def\mapdeg{\mathrm{degMap}}

\def\betas{\mult(\myB(\cS))}

\def\myR{\mathcal{R}}
\def\oz{{\overline z}}
\def\ot{{\overline t}}
\def\oh{{\overline h}}
\def\ox{{\overline x}}

\def\oy{{\overline y}}

\def\proj2{\mathbb{P}^2(\mathbb{K})}
\def\projtres{\mathbb{P}^3(\mathbb{K})}
\def\myV{\mathscr{V}}

\def\myS{\mathscr{S}}
\def\myR{\mathscr{R}}

\def\myB{\mathscr{B}}
\def\myC{\mathscr{C}}

\def\myL{\mathscr{L}}

\def\myG{\mathscr{G}(\projtres)}

\def\myGS{\mathscr{G}(\proj2)}

\def\sing{\mathrm{Sing}}

\def\para{\vspace{2.5mm}}
\def\cP{{\mathcal P}}
\def\cQ{{\mathcal Q}}

\def\gcd{{\rm gcd}}

\def\S{{\mathcal S}}
\def\cS{{\mathcal S}}
\def\cR{{\mathcal R}}

\def\mult{\mathrm{mult}}

\def\deg{{\rm deg}}

\usepackage{mathrsfs}

\usepackage{algorithm}
\usepackage{algorithmic}

\def\proj2{\mathbb{P}^2(\mathbb{K})}
\def\projtres{\mathbb{P}^3(\mathbb{K})}
\def\myV{\mathscr{V}}

\def\myS{\mathscr{S}}
\def\myR{\mathscr{R}}

\def\myB{\mathscr{B}}
\def\myC{\mathscr{C}}

\def\myL{\mathscr{L}}

\def\myG{\mathscr{G}(\projtres)}

\def\myGS{\mathscr{G}(\proj2)}

\def\oz{\,\overline{z}\,}

\def\cL{\mathcal{L}}

\def\cM{\mathcal{M}}
\def\cU{\mathcal{U}}
\def\K{\mathbb{K}}

\def\LcS{{}^{L\!}\cS}
\def\cRL{\cR^{L^{\!-1}}}

\def\ot{\,{\overline t}\,}
\def\ox{\,{\overline x}\,}

\def\oy{\,{\overline y}\,}

\begin{document}

 \title{Computing Birational Polynomial Surface Parametrizations without Base Points}

\author{Sonia P\'erez-D\'{\i}az and J. Rafael Sendra\\
Dpto. de F\'{\i}sica y Matem\'aticas \\
        Universidad de Alcal\'a \\
      E-28871 Madrid, Spain  \\
sonia.perez@uah.es, rafael.sendra@uah.es}
\date{}          
\maketitle

\begin{abstract}
In this paper, we present an algorithm for reparametrizing birational surface parametrizations into birational polynomial surface parametrizations without base points, if they exist. For this purpose, we impose a transversality condition to the base points of the input parametrization.
  \end{abstract}


\para

\section{Introduction}\label{S-intro}
Algebraic surfaces are mainly studied from three different, but related, points of view, namely: pure theoretical, algorithmic and because of their applications. In this paper, we deal with some computational problems of algebraic surfaces taking into account the potential applicability.

\para

 In many different applications, as for instance in geometric design (see e.g. \cite{HL}) parametric representations of surfaces are more suitable than implicit representations. Among the different types of parametric representations, one may distinguish radical parametrizations (see \cite{SS-radical}) and rational parametrizations (see e.g. \cite{SchichoParam}), the first being tuples of fractions of nested radical of bivariate polynomials, and the second being tuples of fractions of bivariate polynomials; in both cases the tuples are with generic Jacobian of rank 2. Other parametric representations by means of series can be introduced, but this is not within the scope of this paper. One may observe that the set of rational parametrizations is a subclass of the class of radical parametrizations. Indeed, in \cite{SSV17}, one can find an algorithm to decide whether a radical parametrization can be transformed by means of a change of the parameters into a rational parametrization; in this case, we say that a reparametrization has been performed.

\para

Now, we consider a third type of parametric representation of the surface, namely, the polynomial parametrization. That is, tuple of bivariate polynomials with generic Jacobian of rank 2. Clearly the class of polynomial parametrizations is a subclass of the class of the  rational parametrizations, and the natural question of deciding whether a given rational parametrization can be reparametrized into a polynomial parametrization appears. This is, indeed, the problem we deal with in the paper. Unfortunately the inclusion of each of these  classes into the
next one is strict, and hence the corresponding reparametrizations  are not always feasible.
In some practical applications, the alternative is to use piecewise parametrizations with the desired property (see e.g. \cite{MJL2017}, \cite{SGBsupport}).

\para

Before commenting the details of our   approach to the problem, let us look at some reasons why polynomial parametrizations may be more interesting than rational ones. In general,  rational parametrizations are dominant over the surface (i.e. the Zariski closure of its image is the surface), but not necessarily surjective. This may
introduce difficulties when applying the parametric representation to a problem, since the answer might be within the non-covered area of the surface. For the curve case, polynomial parametrizations are always surjective (see \cite{Sendra-Normal}). For the surface case, the result is not so direct but there are some interesting results for polynomial parametrizations to be surjective (see \cite{PSV-normal}) as well as subfamilies of polynomial parametrizations that are surjective (see \cite{SSV15}). Another issue that could be mention is the numerical instability when the values, substituted in the parameters of the parametrizations, get close to the poles of the rational functions; note that, in this case, the denominators define algebraic curves which points are all poles of the parametrization. One may also think on the advantages of providing a polynomial parametrization instead of a rational parametrization, when facing surface integrals. Let us mention  a last example of motivation: the algebra-geometric technique for solving non autonomous ordinary differential equations (see \cite{Graseger}, \cite{NW}). In these cases, the differential equation  is seen algebraically and hence representing a surface. Then, under the assumption that this surface is rational (resp. radical) the general rational (resp. radical) solution, if it exists, of the differential equation is determined from a rational parametrization of the surface. This process may be simplified  if the associated algebraic parametrization admits a polynomial parametrization.

\para

Next, let us introduce, and briefly comment, the notion of base points of a rational  parametrization. A base point of a given rational parametrization is a common solution of all numerators and denominators of the parametrization (see e.g. \cite{CoxPerezSendra2020}, \cite{PSbasePoints}). The presence of this type of points is a serious obstacle when approaching many  theoretical, algorithmic or applied questions   related to the surface represented by the parametrizations; examples of this phenomenon can be found in, e.g., \cite{BCD}, \cite{CGZ2000}, \cite{SSV14}, \cite{SG}. In addition, it happens that rational surface may admit, both, birational parametrizations with empty base locus and with non-empty base locus. Moreover, the behavior of the base locus is not controlled, at least to our knowledge, by the existing parametrization algorithms or when the resulting parametrization appears as the consequence of the intersection of higher dimension varieties, or as the consequence of  cissoid, conchoid, offsetting, or any other geometric design process applied to a surface parametrization (see e.g. \cite{ArrondoSS}, \cite{cissoids}, \cite{SSconchoids}, \cite{JanMiroslavConvolution}).

\para

In this paper, we solve the problem, by means of reparametrizations, of computing a birational polynomial parametrization without base points of a rational surface, if it exists. For this purpose, we assume that we are given a birational parametrization of the surface that has the property of being transversal (this is a notion introduced in the paper, see Section \ref{sec-trasnversal-base-locus} for the precise definition). Essentially, the idea of transversality is to assume that the multiplicity of the base   points is minimal. Since, by definition (see Section \ref{sec-trasnversal-base-locus}) this multiplicity is introduced as a multiplicity of intersection of two algebraic curves, one indeed is requiring the transversality of the corresponding tangents. In this paper, we have not approached the problem of eliminating this hypothesis, and we leave it as future work in case it exists.

\para

The general idea to solve the problem is as follows. We are given a birational parametrization $\cP$ and let $\cQ$ be the searched birational polynomial parametrization without base points; let us say, first of all, that throughout the paper we work projectively. Then, there exists a birational map, say $\cS_\cP$, that relates both parametrizations as $\cQ=\cP\circ \cS_\cP$. Then, taking into account that the base locus of $\cS_\cP$ and $\cP$ are the same, that they coincide also in multiplicity, and applying some additional properties on base points stated  in Subsection \ref{subsec-further-results}, we introduce a $2$-dimensional linear system of curves, associated to an effective divisor generated by the base points of $\cP$. Then, using the transversality we prove that every basis of the linear system, composed with a suitable  birational transformation, provides a reparametrization of $\cP$ that yields to a polynomial parametrization with empty base locus.

\para

The structure of the paper is as follows. In Section \ref{sec-preliminary-base-points}, we introduce the notation and we recall some definitions and properties on base points, essentially taken from \cite{CoxPerezSendra2020}. In Section \ref{sec-trasnversal-base-locus} we state some additional required properties on base points, we introduce the notion of transversality of a base locus, both for birational maps of the projective plane and for rational surface projective parametrizations. Moreover, we establish some fundamental properties that require the transversality. Section \ref{section-properPolRep} is devoted to state the theoretical frame for solving the central problem treated in the paper. In
Section \ref{section-algorithms-examples}, we derive the algorithm that is illustrated by means of some examples. We finish the paper with a section on conclusions.

\section{Preliminary on Basic Points and Notation}\label{sec-preliminary-base-points}

\para

In this section, we briefly recall some of the notions related to base points and we introduce some notation; for further results on this topic we refer to \cite{CoxPerezSendra2020}.
We distinguish three  subsections. In Subsection \ref{subsec-notation}, the notation that will be used throughout the paper is introduced. The next subsection focuses on birational surface parametrizations, and the third subsection on birational maps of the projective plane.

\subsection{Notation}\label{subsec-notation}

Let, first of all, start fixing some notation. Throughout this paper, $\K$ is an algebraically closed field of characteristic zero. $\ox=(x_1,\ldots,x_4), \oy=(y_1,\ldots,y_4)$ and $\ot=(t_1,t_2,t_3)$. $\mathbb{F}$ is the algebraic closure of
$\mathbb{K}(\ox,\oy).$ In addition, $\mathbb{P}^{k}(\K)$ denotes the $k$--dimensional projective space, and $\mathscr{G}(\mathbb{P}^{k}(\K))$ is the set of all projective transformations of $\mathbb{P}^{k}(\K)$.

Furthermore, for a rational map
\[ \begin{array}{cccc}
\cM: & \mathbb{P}^{k_1}(\K) & \dashrightarrow & \mathbb{P}^{k_2}(\K) \\
     &     \oh=(h_1:\cdots:h_{k_{1}+1}) & \longmapsto & (m_1(\oh):\cdots: m_{k_2+1}(\oh)),
\end{array}
\]
where the non-zero $m_i$ are homogenous polynomial in $\oh$ of the same degree, we denote by $\deg(\cM)$ the degree $\deg_{\oh}(m_i)$, for $m_i$ non-zero, and by $\mapdeg(\cM)$ the degree of the map $\cM$; that is,  the cardinality of the generic fiber of $\cM$ (see e.g.\ \cite{Harris:algebraic}).

For $L\in \mathscr{G}(\mathbb{P}^{k_2}(\K))$, and $M\in \mathscr{G}(\mathbb{P}^{k_1}(\K))$ we denote the left composition and the right composition, respectively, by
\[ {}^{L}\!\cM:=L\circ \cM, \,\,\, \cM^{M}:=\cM\circ M.\]

 Let $f\in \mathbb{L}[t_1,t_2,t_3]$ be homogeneous and non-zero, where $\mathbb{L}$ is a field extension of $\K$. Then $\myC(f)$ denotes the projective plane curve defined by $f$ over the algebraic closure of $\mathbb{L}$.

 Let $\myC(f),\myC(g)$ be two curves in $\proj2$. For $A\in \proj2$, we represent by
 $\mult_A(\myC(f),\myC(g))$ the multiplicity of intersection of $\myC(f)$ and $\myC(g)$ at $A$. Also, we denote by $\mult(A,\myC(f))$ the multiplicity of $\myC(f)$ at $A$.

Finally, $\myS\subset \projtres$ represents a rational projective surface.

\subsection{Case of surface parametrizations}\label{subsec-param}

In this subsection, we consider  a
rational parametrization of the projective rational surface $\myS$, namely,
\begin{equation}\label{eq-Param}
\begin{array}{llll}
\mathcal{P}: & \mathbb{P}^{2}(\K) &\dashrightarrow &\myS \subset \mathbb{P}^{3}(\K) \\
& \ot & \longmapsto & (p_1(\ot):\cdots:p_4(\ot)),
\end{array}
\end{equation}
where $\ot=(t_1,t_2,t_3)$ and  the $p_i$ are homogenous polynomials of the same degree such that $\gcd(p_1,\ldots,p_4)=1$.

\para

\begin{definition}\label{def-base-point-P}
 A \textsf{base point of $\mathcal{P}$} is an element  $A\in \mathbb{P}^{2}(\K)$ such that $p_i(A)=0$ for every $i\in \{1,2,3,4\}$.  We denote by $\myB({\cP})$ the set of base points of ${\cP}$. That is $\myB(\cP)=\myC(p_1)\cap \cdots \cap \myC(p_4)$.
 \end{definition}

 \para

In order to deal with the base points of the parametrization, we  introduce the  following auxiliary polynomials:
\begin{equation}\label{eq-W}
\begin{array}{lll}
W_1(\ox,\ot):=\sum_{i=1}^{4} x_i\, p_i(t_1,t_2,t_3)
 \\ \noalign{\smallskip}
 W_2(\oy,\ot):=\sum_{i=1}^{4} y_i\, p_i(t_1,t_2,t_3),
 \end{array}
 \end{equation}
 where $x_i, y_i$ are new variables. We will work with the projective  plane curves $\myC(W_{i})$ in $\mathbb{P}^{3}(\mathbb{F})$.  Similarly, for $M=(M_1:M_2:M_3)\in \myG$, we define,
 \begin{equation}\label{eq-W}
\begin{array}{lll}
W_{1}^{M}(\ox,\ot):=\sum_{i=1}^{4} x_i\, M_i(\cP(\ot))
 \\ \noalign{\smallskip}
 W_{2}^{M}(\oy,\ot):=\sum_{i=1}^{4} y_i\,  M_i(\cP(\ot)).
 \end{array}
 \end{equation}

\para

\begin{remark}\label{remark-notation-Wi}
Some times, we will need to specify the parametrization in the polynomials above. In those cases, we will write $W_{i}^{\cP}$ or $W_{i}^{M,\cP}$    instead of $W_i$ or $W_{i}^{M}$; similarly, we may write $\myC(W_{1}^{\cP})$ and $\myC(W_{1}^{M,\cP})$.
\end{remark}

\para

 Using the multiplicity   of
intersection of these two curves, we define the multiplicity of a base point as follows.

\para

\begin{definition}\label{def-multiplicity-BP}
The \textsf{multiplicity of a
base point} $A\in \myB(\cP)$ is  $\mult_{A}(\myC(W_1),\myC(W_2))$, that is, is the multiplicity of intersection at $A$ of $\myC(W_1)$ and $\myC(W_2)$; we denote it by \begin{equation}\label{eq-multBasePointP}
\mult(A,\myB(\cP)):=\mult_{A}(\myC(W_1),\myC(W_2))
\end{equation}
In addition, we define the \textsf{multiplicity of the base points locus of ${\cP}$},  denoted $\mult(\myB(\cP))$, as
\begin{equation}\label{eq-multBP}
\mult(\myB(\cP)):={\sum_{A\in \myB({\cP})} \mult(A,\myB(\cP))=}
\sum_{A\in \myB({\cP})} \mult_A(\myC(W_{1}),\myC(W_{2})).
\end{equation}
Note that, since $gcd(p_1,\ldots,p_4)=1$, the set $\myB(\cP)$ is either empty of finite.
\end{definition}

\para

\para

\noindent
For the convenience of the reader we recall here some parts of Proposition 2 in \cite{CoxPerezSendra2020}.

\para

\begin{lemma}\label{lemma-BasePointP}
If $L\in \myG$, then:
\begin{enumerate}
\item If $A\in \myB({\cP})$, then $$\mult(A,\myC(W_{1}^{L}))=\mult(A,\myC(W_{2}^{L}))=\min\{\mult(A,\myC(p_i))\,|\, i=1,\ldots,4\}.$$
\item If $A\in \myB({\cP})$, then the tangents to $\myC(W_{1}^{L})$ at $A$ {\rm(}{\hskip-1pt}similarly to $\myC(W_{2}^{L})${\rm)}, with the corresponding multiplicities, are the factors in $\mathbb{K}[\ox,\ot]\setminus \mathbb{K}[\ox]$ of
    \[ \epsilon_1 x_1 T_1+\epsilon_2 x_2 T_2 +\epsilon_3 x_3 T_3+\epsilon_4 x_4 T_4,\]
    where $T_i$ is the product of the tangents, counted with multiplicities, of $\myC(L_i({\cP}))$ at $A$, and where $\epsilon_i=1$ if $\mult(A,\myC(L_i({\cP}))))=\min\{\mult(A,\myC(L_i({\cP})))\,|\, i=1,\ldots,4\}$ and $0$ otherwise.
\end{enumerate}
\end{lemma}

\subsection{Case of rational maps of $\proj2$}\label{subsec-proj2}

\para
In this subsection, let
\begin{equation}\label{eq-S}
\begin{array}{cccc}
\cS: & \proj2 &  \dashrightarrow & \proj2   \\
 &\ot=(t_1:t_2:t_3) & \longmapsto & \cS(\ot)=(s_1(\ot):s_2(\ot):s_3(\ot)),
\end{array}
\end{equation}
where  $\gcd(s_1,s_2,s_3)=1$, be a dominant rational transformation of $\proj2$.

\para

\begin{definition}
  $A\in \proj2$ is a \textsf{base point} of $\cS(\ot)$ if $s_1(A)=s_2(A)=s_3(A) = 0$. That is, the
base points of $\cS$ are the intersection points of the projective plane curves, $\myC(s_i)$, defined over $\mathbb{K}$ by $s_i(\ot)$, $i=1,2,3$. We denote by $\myB(\cS)$ the set of base points of $\cS$.
\end{definition}

\para

\noindent
We introduce the polynomials
\begin{equation}\label{eq-V}
\begin{array}{l}
 V_{1}=\sum_{i=1}^{3} x_i\, s_i(\ot) \in \mathbb{K}(\ox,\oy)[\ot] \\
 \noalign{\medskip}
 V_{2}=\sum_{i=1}^{3} y_i \,s_i(\ot)\in
 \mathbb{K}(\ox,\oy)[\ot],
 \end{array}
\end{equation}
where $x_i,y_j$ are new variables and we consider the curves $\myC(V_{i})$ over the field $\mathbb{F}$; compare with \eqref{eq-W}. Similarly, for every $L\in \myGS$  we introduce the polynomials
\begin{equation}\label{eq-VL}
\begin{array}{l}
 V_{1}^{L}=\sum_{i=1}^{3} x_i\,L_i(\cS) \in \mathbb{K}(\ox,\oy)[\ot] \\
 \noalign{\medskip}
 V_{2}^{L}=\sum_{i=1}^{3} y_i \,L_i(\cS)\in
 \mathbb{K}(\ox,\oy)[\ot],
 \end{array}
\end{equation}

\para

\begin{remark}\label{remark-notation-Vi}
Some times, we will need to specify the rational map in the polynomials above. In those cases, we will write $V_{i}^{\cS}$ or $V_{i}^{L,\cS}$    instead of $V_i$ or $V_{i}^{L}$; similarly, we may write $\myC(V_{1}^{\cS})$ and $\myC(V_{1}^{L,\cS})$.
\end{remark}

\para

As we did in Subsection \ref{subsec-param}, we have the following notion of multiplicity.

\para

\begin{definition}\label{definition-BS}
For $A\in \myB(\cS)$, we define the \textsf{multiplicity of intersection of $A$}, and we denote it by $\mult(A,\myB(\cS))$, as
\begin{equation}\label{eq-DefmultPbaseS}
\mult(A,\myB(\cS)):=\mult_{A}(\myC(V_{1}),\myC(V_2)).
\end{equation}
In addition, we define the \textsf{multiplicity of the base points locus of ${\cS}$}, denoted $\betas$, as (note that, since $\gcd(s_1,s_2,s_3)=1$, $\myB(\cS)$ is either finite or empty)
\begin{equation}\label{eq-DefmultBS}
\betas:=
{\sum_{A\in \myB({\cS})} \mult(A,\myB(\cS))=}\sum_{A\in \myB({\cS})} \mult_A(\myC(V_{1}),\myC(V_{2}))
\end{equation}
\end{definition}

\para

\para

Next result is a direct extension of Proposition 2 in \cite{CoxPerezSendra2020} to the case of birational transformation of $\proj2$.

\para

\begin{lemma}\label{lemma-BasePointS}
If $L\in \myGS$ then
\begin{enumerate}
\item $\myB(\cS)=\myC(V_{1}^{L}) \cap \myC(V_{2}^{L})\cap \proj2$.
\item Let $A\in \myB(\cS)$ then $$\mult(A,\myC(V_{1}^{L}))=\mult(A,\myC(V_{2}^{L}))=\min\{\mult(A,\myC(s_i))\,|\, i=1,2,3\}.$$
\item Let $A\in \myB(\cS)$. The tangents to $\myC(V_{1}^{L})$ at $A$ (similarly to $\myC(V_{2}^{L})$), with the corresponding multiplicities, are the factors in $\mathbb{K}[\ox,\ot]\setminus \mathbb{K}[\ox]$ of
    \[ \epsilon_1 x_1 T_1+\epsilon_2 x_2 T_2 +\epsilon_3 x_3 T_3,\]
    where $T_i$ is the product of the tangents, counted with multiplicities, of $\myC(L_i(\cS))$ at $A$, and where $\epsilon_i=1$ if $\mult(A,\myC(L_i(\cS))))=\min\{\mult(A,\myC(L_i(\cS)))\,|\, i=1,2,3\}$ and 0 otherwise.
\end{enumerate}
\end{lemma}

\section{Transversal Base Locus}\label{sec-trasnversal-base-locus}

\para

In this section, we present some new  results on base points that complement those in \cite{CoxPerezSendra2020} and we introduce and analyze the notion of transversality in conexion with the base locus.

\para
Throughout this section, let $\cS=(s_1:s_2:s_3)$, with $\gcd(s_1,s_2,s_3)=1$, be as in \eqref{eq-S}. In addition, in the sequel, we assume that $\cS$  is birational. Let the inverse  of $\cS$ be denoted  by $\cR=(r_1:r_2:r_3)$; that is $\cR:=\cS^{-1}$. Also, we consider a rational surface parametrization $\cP=(p_1:\cdots:p_4)$, with $\gcd(p_1,\ldots,p_4)=1$, be as in \eqref{eq-Param}. We assume that $\cP$ is birational.

\subsection{Further results on base points}\label{subsec-further-results}

We start analyzing the rationality of the curve $\myC(V_{i}^{L})$ (see \eqref{eq-VL}).

\para

\begin{lemma}\label{lemma-RatV1V2}  There exists a non-empty open subset $\Omega_1$ of $\myGS$ such that if $L \in \Omega_1$ then $\myC(V_{1}^{L})$ is a rational curve. Furthermore,
\[ \cV_1(\ox,h_1,h_2)=\cR^{L^{-1}}(h_1x_3,h_2x_3,-(h_1x_1+x_2h_2)) \]
is a birational parametrization of $\myC(V_{1}^{L})$.
\end{lemma}
\begin{proof} We start proving that for every $L\in \myGS$, $V_{1}^{L}$ is irreducible. Indeed, let
$L=(\sum \lambda_i t_i: \sum \mu_i t_i: \sum \gamma_i t_i)\in \myGS$. Then $V_{1}^{L}=
(\lambda_1 x_1+\mu_1 x_2+ \gamma_1 x_3) s_1+(\lambda_2 x_1+\mu_2 x_2+ \gamma_2 x_3) s_2+
(\lambda_3 x_1+\mu_3 x_2+ \gamma_3 x_3) s_3$.  $\gcd(s_1,s_2,s_3)=1$ and
$\gcd(\lambda_1 x_1+\mu_1 x_2+ \gamma_1 x_3, \lambda_2 x_1+\mu_2 x_2+ \gamma_2 x_3, \lambda_3 x_1+\mu_3 x_2+ \gamma_3 x_3)=1$ because the determinant of the matrix associated to $L$ is non-zero. Therefore, $V_{1}^{L}$ is irreducible.

In the following, to define the open set $\Omega_1$, let $\cL(t_1,t_2,t_3)=(\cL_1:\cL_2:\cL_3)$ be a generic element of $\myGS$; that is, $\cL_i=z_{i,1} t_1+ z_{i,2} t_2+ z_{i,3} t_3$, where $z_{i,j}$ are undetermined coefficients satisfying that the determinant of the corresponding matrix is not zero. Furthermore, for     $L\in \myGS$, we denote by $\oz^L$ the coefficient list of $L$.
We also introduce the polynomial $R^{\cL}=x_{1} \cL^1+x_{2} \cL^{2} +x_3 \cL^3=(\sum z_{i,1}x_i) t_1+(\sum z_{i,2}x_i) t_2+(\sum z_{i,3}x_i) t_3$. Similarly, for $L\in \myGS$, we denote $R^{L}=R_{\cL}(\oz^L,\ox,\ot)$.

 We consider the birational extension $\cR_{\ox}: \mathbb{P}^{2}(\mathbb{F}) \dashrightarrow  \mathbb{P}^{2}(\mathbb{F})$ of $\cR$ from $\proj2$ to $\mathbb{P}^{2}(\mathbb{F})$. Let $\cU_{\ox} \subset \mathbb{P}^{2}(\mathbb{F}) $ be the  open set where the $\cR_{\ox}$ is bijective; say that $\cU_{\ox}=\mathbb{P}^{2}(\mathbb{F}) \setminus \Delta$. We express the close set $\Delta$ as $\Delta=\Delta_1\cup \Delta_2$ where $\Delta_1$ is either empty or it is a union of finitely many curves, and $\Delta_2$ is either empty or finite many points. We fix our attention in $\Delta_1$. Let $f(\ot)$ be the defining polynomial of $\Delta_1$. Let $Z(\oz,t_1,t_2)$ be the remainder of $f$ when diving by $V_{1}^{\cL}$ w.r.t $t_3$. Note that $R^{\cL}$ does not divide $f$ since $R^{\cL}$ is irreducible and depends on $\oz$. Hence $Z$ is no zero. Let $\alpha(\oz)$ be the numerator of a non-zero coefficient of $Z$ w.r.t. $\{t_1,t_2\}$ and let $\beta(\oz)$ the l.c.m. of the denominators of all coefficients of $Z$ w.r.t. $\{t_1,t_2\}$. Then,
we define $\Omega_1$ as
\[ \Omega_{1}=\{L\in \myGS \,|\, \alpha(\oz^L)\beta(\oz^L)\neq 0\} \]
We observe that, by construction, if $L\in \Omega_1$ then $\myC(R^L)\cap \cU_{\ox}$ is  dense in
$\myC(R^L)$.

Let $\overline{a},\overline{b}\in \myC(R^L)\cap \cU_{\ox}$ be two different points, then by injectivity $\cR_{\ox}(\myC(R^L))$ contains at least two points, namely $\cR_{\ox}(\overline{a})$ and $\cR_{\ox}(\overline{b})$. In this situation, since $\cR_{\ox}(\myC(R^L))$ and $\myC(V_1)$ are irreducible we get that $\overline{\cR_{\ox}(\myC(R^L))}=\myC(V_1)$, and hence $\myC(V_1)$ is a rational curve $\mathbb{P}^2(\mathbb{F})$.  Furthermore, one easily may check that $\cV_1$ parametrizes $\myC(V_{1})$  and it is proper since  $\cR$ is birational.
\end{proof}

\para

 \begin{remark}\label{rem-V1V2}
   Note that $\cV_1(t_3,0,-t_1,t_1, t_2)=\cRL(-t_1t_1,-t_1t_2,-t_1t_3)=\cRL(\ot)$. Hence
   \[ \LcS(\cV_1(t_3,0,-t_1,t_1, t_2))=\LcS(\cRL(\ot))=(t_1:t_2:t_3). \]
 Therefore,
   \[L_i(\cS(\cV_1(t_3,0,-t_1,t_1, t_2))=t_i\cdot \wp(\ot),\,i=1,2,3.\]
 \end{remark}

\para

Next Lemma analyzes the rationality of the curves $\myC(L_i(\cS))$ where $L=(L_1:L_2:L_3)\in \myGS$.

\para

\begin{lemma}\label{lemma-rat-Si}
There exists a non-empty Zariski open subset $\Omega_2$ of $\myGS$ such that if $L\in \myGS$ then  $\myC(L_i(\cS))$, where $i\in\{1,2,3\}$, is rational.
\end{lemma}
\begin{proof}
Let $\cU$ be the open subset where $\cR$ is a bijective map, and let
$\{ \rho_{j,1}t_1+\rho_{j,2} t_2+\rho_{j,3} t_3\}_{j=1,\dots,n}$ be the linear forms defining the lines, if any, included in $\proj2\setminus \cU$. Then, we take $\Omega_2=\cap_{j=1}^{n} \Sigma_j$
where
\[ \Sigma_j=\left\{\left. \left(\sum \lambda_i t_i: \sum \mu_i t_i: \sum \gamma_i t_i\right)\in \myGS \,\right|\,
\begin{array}{l} (\lambda_{1}:\lambda_{2}:\lambda_{3})\neq (\rho_{j,1}:\rho_{j,2}:\rho_{j,3}), \\
(\mu_{1}:\mu_{2}:\mu_{3})\neq (\rho_{j,1}:\rho_{j,2}:\rho_{j,3}), \\
(\gamma_{1}:\gamma_{2}:\gamma_{3})\neq (\rho_{j,1}:\rho_{j,2}:\rho_{j,3})
\end{array} \right\}.
\]
Now, let $L=(L_1:L_2:L_3)\in \Omega_2$. By construction, $\myC(L_i)\cap \cU$ is dense in $\myC(L_i)$, for $i\in \{1,2,3\}$. In this situation, reasoning as in the last part of the proof of Lemma \ref{lemma-RatV1V2}, we get that $\overline{\cR(\myC(L_i))}=\myC(L_i(\cS))$. Therefore, $\myC(L_i(\cS))$ is rational.
\end{proof}

\para

\noindent
The following lemma follows from Lemma \ref{lemma-BasePointS}.

\para

\begin{lemma}\label{lemmaBsBLs}
If $L\in \myGS$ then
\begin{enumerate}
\item  $\myB(\cS)=\myB(\LcS)$.
\item For $A\in \myB(\cS)$ it holds that $\mult(A,\myB(\cS))=\mult(A,\myB(\LcS))$.
\item $\mult(\myB(\cS))=\mult(\myB(\LcS))$.
\end{enumerate}
\end{lemma}
\begin{proof}
(1) Let $A\in \myB(\cS)$ then $s_1(A)=s_2(A)=s_3(A)=0$. Thus, $L_1(\cS)(A)=L_2(\cS)(A)=L_3(\cS)(A)=0$. So,
$A\in \myB(\LcS)$. Conversely, let $A\in \myB(\LcS)$. Then expressing $L(\cS(A))=\overline{0}$ in terms of matrices, since $L$ is invertible, we have that $\cS(A)=L^{-1}(0,0,0)=(0,0,0)$. Thus, $A\in \myB(\cS)$. \\
(2) and (3) follows from Theorem 5 in \cite{CoxPerezSendra2020}. 
\end{proof}

\para

\para

\para

\begin{lemma}\label{lemma-singularlocus} There exists a non-empty Zariski open subset $\Omega_3$ of $\myGS$ such that if $L\in  \Omega_3$ then for every $A\in \myB(\cS)$ it holds that
 $$\mult(A,\myC(V_{1}^{L}))=\mult(A,\myC(L_1(\cS)))=\mult(A,\myC(L_2(\cS)))=\mult(A,\myC(L_3(\cS))).$$
  \end{lemma}
\begin{proof}
Let $A\in \myB(\cS)$. Then, by Lemma \ref{lemma-BasePointS} (2), we have that
\begin{equation}\label{eq-mA} m_A:=\mult(A,\myC(V_{1}^{L}))=\min\{ \mult(A,\myC(s_i)\,|\,i\in \{1,2,3\} \}, \,\,\forall \,\, L\in \myGS.
\end{equation}
Let us assume w.l.o.g. that the minimum above is reached for $i=1$. Then all $(m_A-1)$--order derivatives of the forms  $s_i$ vanish at $A$, and there exists an $m_A$--order partial derivative of $s_1$ not vanishing at $A$. Let us denote this partial derivative as $\partial^{m_A}$.

Now, let $\cL$ be as in the proof of Lemma \ref{lemma-RatV1V2}.  Then,
\[ g_i(\oz):=\partial^{m_A} \cL_i(\cS)(A)=z_{i,1} \partial^{m_A} s_1 (A)+ z_{i,2} \partial^{m_A} s_2(A)+ z_{i,3} \partial^{m_A} s_3(A)\in \K[\oz] \]
is a non-zero polynomial because $\partial^{m_A} s_1 (A)\neq 0$. We then consider the open subset (see
proof of Lemma \ref{lemma-RatV1V2} for the notation $\oz^{L}$)
\[ \Omega_{A}=\{ L\in \myGS\,|\, g_1(\oz^L)g_2(\oz^L)g_3(\oz^L)\neq 0\}\neq \emptyset.  \]
In this situation, we take
\[ \Omega_3=\bigcap_{A\in \myB(\cS)} \Omega_{A} \]
Note that, since $\myB(\cS)$ is finite then $\Omega_3$ is open. Moreover, since $\myGS$ is irreducible then $\Omega_3$ is not empty.

Let us prove that $\Omega_3$ satisfies the property in the statement of the lemma. Let $L\in \Omega_3$ and $A\in \myB(\cS)$. Let $m_A$ be as in \eqref{eq-mA}.
Then all partial derivatives of $L_i(\cS)$, of any order smaller than $m_A$, vanishes at $A$. Moreover, since $L\in \Omega_3\subset  \Omega_A$, it holds that $\partial^{m_A} L_i(\cS)(A)\neq 0$ for $i=1,2,3$.
Therefore,
\[ \mult(A,\myC(V_{1}^{L}))=m_A=\mult(A,\myC(L_1(\cS))=\mult(A,\myC(L_2(\cS))=\mult(A,\myC(L_3(\cS)) \]
\end{proof}

\para

\begin{remark}\label{rem-BPBPL}
We note that the proofs of Lemma \ref{lemmaBsBLs} and Lemma \ref{lemma-singularlocus}
are directly adaptable to the case of birational surface parametrizations. So, both lemmas hold  if  $M\in \myG$ and we replace $\cS$ by the birational surface parametrizaion $\cP$ and $\LcS$ by
${}^{M}\!\cP:=M\circ \cP$.
\end{remark}

\para

In the following, we denote by $\sing({\cal D})$, the set of singularities of an algebraic plane curve $\cal D$.

\para

\begin{corollary}\label{cor-tgs}
Let $\Omega_3$ be the open subset in Lemma \ref{lemma-singularlocus} and $L\in \Omega_3$. It holds that
\begin{enumerate}
\item $\cap_{i=1}^{3} \sing(\myC(L_i(\cS))) \cap \myB(\cS)\subset \sing(\myC(V_{1}^{L}))$.
\item Let $A\in \myB(\cS)$. The tangents to $\myC(V_{1}^{L})$ at $A$, with the corresponding multiplicities, are the factors in $\mathbb{K}[\ox,\ot]\setminus \mathbb{K}[\ox]$ of
    \[  x_1 T_1+x_2 T_2 + x_3 T_3,\]
    where $T_i$ is the product of the tangents, counted with multiplicities, to $\myC(L_i(\cS))$ at $A$.
\item Let $A\in \myB(\cS)$, and let $T_i$ be the product of the tangents, counted with multiplicities, to $\myC(L_i(\cS))$, at $A$. If $\gcd(T_1,T_2,T_3)=1$, then
    \[ \mult_A(\myC(V_{1}^{L}),\myC(V_{2}^{L}))=\mult(A,\myC(L_i(\cS))^2,\,\, i\in \{1,2,3\} \]
\end{enumerate}
\end{corollary}
\begin{proof} \
(1)  Let $A\in \cap_{i=1}^{3} \sing(\myC(L_i(\cS))) \cap \myB(\cS)$. By Lemma \ref{lemma-singularlocus}, $m:=\mult(A,\myC(L_i(\cS)))>0$, for $i\in \{1,2,3\}$, and $\mult(A,\myC(V_{1}^{L}))=m>0$. So, $A\in \sing(\myC(V_{1}^{L}))$. \\
(2) follows from Lemma \ref{lemma-singularlocus} and Lemma \ref{lemma-BasePointS}. \\
(3) By (2) the tangents to $\myC(V_{1}^{L})$ and to $\myC(V_{2}^{L})$ at $A$ are $\mathcal{T}_1:=\sum x_i T_i$ and $\mathcal{T}_2:=\sum y_i T_i$, respectively. Since $\gcd(T_1,T_2,T_3)=1$, then $\mathcal{T}_i$ is primitive, and hence $\gcd(\mathcal{T}_1,\mathcal{T}_2)=1$. That is, $\myC(V_{1}^{L})$ and $\myC(V_{2}^{L})$ intersect transversally at $A$. From here, the results follows.
 \end{proof}

\subsection{Transversality}

We start introducing the notion of transversality for birational maps of $\proj2$.

\para

\begin{definition}\label{def-transversal}
 We say that $\cS$ is \textsf{transversal} if either $\myB(\cS)=\emptyset$ or
 for every $A\in \myB(\cS)$ it holds that (see \eqref{eq-V})
\[ \mult(A,\myB(\cS))=\mult(A,\myC(V_1))^2 \]
In this case, we also say that the \textsf{base locus of $\cS$ is transversal}.
\end{definition}

\para

{In the following  lemma, we see that the transversality is invariant under left composition with elements in $\myGS$.

\para

\begin{lemma}\label{lemma-transSSL}
If $\cS$ is transversal, then for every  $L\in \myGS$ it holds that $\LcS$ is transversal.
\end{lemma}
\begin{proof}
By Lemma   \ref{lemmaBsBLs} (1), $\myB(\cS)=\myB(\LcS)$.  So, if  $\myB(\cS)=\emptyset$, there is nothing to prove. Let $A\in \myB(\cS)\neq \emptyset$, and let
 $L:=(L_1:L_2:L_3).$   Then
\[ \begin{array}{cclr}
\mult(A,\myB(\LcS))& =& \mult(A,\myB(\cS)) & \text{(see Lemma \ref{lemmaBsBLs} (2))} \\
& =& \mult(A,\myC(V_1))^2 & \text{($\cS$ is transversal)}  \\
& =& \mult(A,\myC(V_{1}^{L,\cS}))^2 & \text{(see Lemma \ref{lemma-BasePointS} (2) and Remark \ref{remark-notation-Vi})}  \\
\end{array}\]
Therefore, $\LcS$ is transversal.
\end{proof}

\para

The next lemma characterizes the transversality by means of the tangents of $\myC(s_i)$ at the base points.

\para
\begin{lemma}\label{Lemma-trasgcd} The following statements are equivalent
\begin{enumerate}
\item $\cS$ is transversal.
\item For every $A \in \myB(\cS)$ it holds that $\gcd(T_{1}, T_{2}, T_3)=1$, where $T_i$ is the product of the tangents, counted with multiplicities, to $\myC(s_i)$ at $A$.
\end{enumerate}
\end{lemma}
\begin{proof}  If $\myB(\cS)=\emptyset$, the result if trivial. Let $\myB(\cS)\neq \emptyset$.  First of all, we observe that, because of Lemma \ref{lemma-transSSL}, we may assume w.l.o.g. that Lemma \ref{lemma-singularlocus}  applies to $\cS$. So, by Definition \ref{def-transversal}, $\cS$ is transversal  if and only if for every $A \in \myB(\cS)$ it holds that
\[\mult(A, \myB(\cS))= \mult(A, \myC(V_1))^2,\]
and, by Definition  \ref{def-multiplicity-BP}, if and only if
\[\mult(A, \myC(V_1))^2=\mult_A(\myC(V_{1}),\myC(V_{2})).\]
 Furthermore, using Theorem 2.3.3 in \cite{SWP}, we have that
\[\mult_A(\myC(V_{1}),\myC(V_{2}))=\mult(A,\myC(V_{1}))\,\mult(A,\myC(V_{2}))\] if and only $V_1$ and $V_2$ intersect transversally
at $A$ i.e. if the curves have no common tangents at $A$ which is equivalent to $\gcd(T_1, T_2, T_3)=1$.
The proof finishes taking into account that, by Lemma  \ref{lemma-singularlocus} $\mult(A,\myC(V_{1}))=\mult(A,\myC(V_{2}))$.
\end{proof}

\para

In the last part of this section, we analyze the relationship of the transversality of a birational map of the projective plane and the transversality of a birational projective surface parametrization. For this purpose, first we introduce the notion of transversality for parametrizations.

\para

\begin{definition}\label{def-transversalP}
Let $\cP$ be a birational surface parametrization of $\projtres$.  We say that $\cP$ is \textsf{transversal} if either $\myB(\cP)=\emptyset$ or for every $A\in \myB(\cP)$ it holds that (see \eqref{eq-W})
\[ \mult(A,\myB(\cP))=\mult(A,\myC(W_1))^2 \]
In this case, we say that the \textsf{base locus of $\cP$ is transversal}.
\end{definition}

\para

\noindent
We start with some technical lemmas.

\para

\begin{lemma}\label{lemma-transPMP}
If $\cP$ is transversal, then for every  $M\in \myGS$ it holds that ${}^{M}\!\cP$ is transversal.
\end{lemma}
\begin{proof}
 The proof is analogous to the proof of Lemma \ref{lemma-transSSL}.
\end{proof}

\para

\para
\begin{lemma}\label{Lemma-traPgcd}  The following statements are equivalent
\begin{enumerate}
\item $\cP$ is transversal.
\item For every $A \in \myB(\cP)$ it holds that $\gcd(T_{1}, \ldots, T_4)=1$, where $T_i$ is the product of the tangents, counted with multiplicities, to $\myC(p_i)$ at $A$.
\end{enumerate}
\end{lemma}
\begin{proof}
 The proof is analogous to the proof of Lemma \ref{Lemma-trasgcd}.
\end{proof}

\para

 The following lemma focusses on the behavior of the base points of $\cP$ when right composing with elements in $\myGS$.

\para

\begin{lemma}\label{lemma-basePbasePL}
Let $L\in \myGS$. It holds that
\begin{enumerate}
\item  $\myB(\cP)=L(\myB(\cP^{L}))$. Furthermore, $A\in \myB(\cP)$ if and only $L^{-1}(A)\in \myB(\cP^L)$.
\item  For $A\in \myB(\cP)$, $\mult(A,\myB(\cP))=\mult(L^{-1}(A),\myB(\cP^{L}))$
\item $\mult(\myB(\cP))=\mult(\myB(\cP^{L}))$.
\item If $\cP$ is transversal then $\cP^L$ is also transversal.
\end{enumerate}
\end{lemma}
\begin{proof}
\noindent (1)   $A\in \myB(\cP)$ iff  $p_i(A)=0$ for $i\in \{1,\ldots,4\}$ iff $p_i(L(L^{-1}(A)))=0$ for $i\in \{1,\ldots,4\}$ iff $L^{-1}(A)\in \myB(\cP^L)$ iff $A\in L(\myB(\cP^L))$. So (1) follows.

\noindent We consider the curves $\myC(W_{i}^{\cP})$ and $\myC(W_{i}^{\cP^L})$ (see Remark \ref{remark-notation-Wi}), and we note that $\myC(W_{i}^{\cP^L})$ is the transformation of $\myC(W_{i})$ under the birational transformation $L^{-1}$ of $\proj2$. Now, (2) and (3) follow from Definition  \ref{def-multiplicity-BP}, and (4) from Lemma \ref{Lemma-traPgcd}.
\end{proof}

\para

Next results  analyze the base loci of  birational surface parametrizations assuming that there exists one of them with empty base locus.

\para

\begin{lemma}\label{lemma-BSBP}  Let $\cP$ and $\cQ$ be two birational projective parametrizations of the same surface $\myS$ such that $\cQ(\cS)=\cP$ and $\myB(\cQ)=\emptyset$.  It holds that
\begin{enumerate}
\item $\myB(\cS)=\myB(\cP)$.
\item If $A\in \myB(\cS)$ then $\deg(\myS)\,\mult(A,\myB(\cS))=\mult(A,\myB(\cP))$.
\end{enumerate}
\end{lemma}
\begin{proof}
Since $\cQ=\cP(\cS)$ and $\myB(\cQ)=\emptyset$, by Theorem 11 in \cite{CoxPerezSendra2020} we get that $\myB({}^{L_{\cS}}\!\cS)=\myB({}^{L_{\cP}}\!\cP)$ for $L_{\cS}$ in a certain open subset of $\myGS$ and $L_{\cP}$ in a certain open subset of $\myG$. Now, using Lemma \ref{lemmaBsBLs}, and Remark \ref{rem-BPBPL} one concludes the proof of statement (1). Statement (2) follows from Theorem 11 in \cite{CoxPerezSendra2020}, taking into account that $\cQ$ is birational.
\end{proof}

\begin{lemma}\label{L-mult} Let $\cP$ and $\cQ$ be two birational projective parametrizations  of the same surface $\myS$ such that $\cQ(\cS)=\cP$ and $\myB(\cQ)=\emptyset$. Then,   for every $A \in \myB(\cS)$ it holds that  (see \eqref{eq-W} and \eqref{eq-V})
$$\mult(A, \myC(W_1))= \mult(A, \myC(V_1))\,\deg(\cQ).$$
\end{lemma}
\begin{proof}
Let $\cP=(p_1:\cdots:p_4)$, and $\cQ=(q_1:\cdots:q_4)$, where $\gcd(q_1,\ldots,q_4)=1$. We know that $p_i=q_i(\cS)$. Moreover,  since $\myB(\cQ)=\emptyset$, by Theorem 10 in \cite{CoxPerezSendra2020}, we have that $\gcd(p_1,\ldots,p_4)=1$.

We start observing that because of  {Lemma} \ref{lemma-BSBP} one has that $\myB(\cS)=\myB(\cP)$. Now, let us consider $L\in \myGS$ and $M\in \myG$. Let $\cQ^*={}^{M}\!Q^{L^{-1}}$, $\cS^*=\LcS$ and $\cP^*={}^{M}\!\cP$. Note that $\cQ^*(\cS^*)=\cP^*$. Moreover, $\myB(\cQ^*)=\emptyset$. Indeed: if $A\in \projtres$ then
$B:=L^{-1}(A)\in \projtres$ and, since $\myB(\cQ)=\emptyset$, $C:=\cQ(B)\in \projtres$. Therefore $\cQ^*(A)=M(B)\in \projtres$ and,
 in consequence, $\myB(\cQ^*)=\emptyset.$  Moreover, $\cQ^*$ and $\cP^*$ parametrize the same surface. Furthermore, by Lemma \ref{lemma-transSSL}, $\cS^*$ is transversal. Thus, $\cS^*,\cP^*,\cQ^*$ satisfy the hypotheses of the lemma. On the other hand,
 by Lemma \ref{lemmaBsBLs} and Remark \ref{rem-BPBPL}, we have that $\myB(\cS^*)=\myB(\cS)=\myB(\cP)=\myB(\cP^*)$. Furthermore,
 by Lemmas \ref{lemma-BasePointP}, \ref{lemma-BasePointS} we have that
 $\mult(A,\myC(V_1))=\mult(A,\myC(V_{1}^{L}))$ and $\mult(A,\myC(W_1))=\mult(A,\myC(W_{1}^{M}))$. Therefore, by   Lemma \ref{lemma-singularlocus} and Remark \ref{rem-BPBPL}, we can assume w.lo.g. that
 for every $A\in \myB(\cS)=\myB(\cP)$ it holds that
 \begin{equation}\label{eqSP}
 \begin{array}{l}
 \text{$\mult(A,\myC(V_1))=\mult(A,\myC(s_i))$ for $i\in \{1,2,3\}$} \\
 \text{$\mult(A,\myC(W_1))=\mult(A,\myC(p_i))$ for $i\in \{1,2,3,4\}$}
 \end{array}
 \end{equation}
   Now, let $A\in \myB(\cP)$ and let $m:=\mult(A,\myC(V_1))$.
   We can  assume w.l.o.g that $A=(0:0:1)$. Let  $T_{i}$ denote the product of the tangents to $s_i$ at $A$. Also, let $\deg(\cS)=\mathfrak{s}$ $\deg(\cP)=\mathfrak{p}$, and  $\deg(\cQ)=\mathfrak{q}$ Then, by \eqref{eqSP}, we may write:
\begin{equation}\label{eqS}
s_i=T_i t_3^{\mathfrak{s}-m}+g_{{m+1,i}}t_3^{\mathfrak{s}-m-1}+\cdots+g_{{\mathfrak{s},i}}
\end{equation}
where $g_{{j,i}}(t_1, t_2)$ are homogeneous forms of degree $j$. In addition, let $q_i$ be expressed as
\begin{equation}\label{eq-q}
q_i(\ot)=F_{\mathfrak{q},i}+F_{\mathfrak{q}-1,i}t_3+\cdots +F_{\ell_i,i} t_{3}^{\mathfrak{q}-\ell_i}.
\end{equation}
where $F_{j,i}(t_1,t_2)$ are homogeneous forms of degree $j$. Then
\[ p_i(\ot)=F_{\mathfrak{q},i}(s_1,s_2)+F_{\mathfrak{q}-1,i}(s_1,s_2)s_3+\cdots +F_{\ell_i,i}(s_1,s_2) s_{3}^{\mathfrak{q}-\ell_i}. \]
Using this expression and \eqref{eqS} can be expressed as
\begin{equation}\label{eq-psubs}
\begin{array}{c}
p_i(\ot)=\left(F_{\mathfrak{q},i}(T_1,T_2)+F_{\mathfrak{q}-1,i}(T_1,T_2) T_3+
\cdots +F_{\ell_i,i}(T_1,T_2) T_{3}^{\mathfrak{q}-\ell_i}  \right) t_{3}^{\mathfrak{q}(\mathfrak{s}-m)} \\  \hspace*{1cm} + \left( \text{terms of degree in ${t_3}$  strictly smaller than $\mathfrak{q}(\mathfrak{s}-m)$}\right).
\end{array}
\end{equation}
Let
\[
H_i:=F_{\mathfrak{q},i}(T_1,T_2)+F_{\mathfrak{q}-1,i}(T_1,T_2) T_3+
\cdots +F_{\ell_i,i}(T_1,T_2) T_{3}^{\mathfrak{q}-\ell_i}. \]
Now, let us prove that $H_i$ is not identically zero. We first observe that $H_i=q_i(T_1,T_2,T_3)$ for   $i\in \{1,2,3,4\}$. We also note that
if there exists $i\in \{1,2,3,4\}$ such that $H_i =0$, by \eqref{eqSP},  it must happen that  for all $i\in \{1,2,3,4\}$ it holds that $ H_i =0$.
Let $H_1$ be zero. Then,
$\mathcal{T}=(T_1:T_2:T_3)\notin \proj2$, because otherwise $\mathcal{T}\in \myB(\cQ)$ and $\myB(\cQ)=\emptyset$. Therefore, $\mathcal{T}$ is a curve parametrization. Thus, if $H_i =0$,  then $\mathcal{T}$ parametrizes a common component of the four curves $\myC(p_i)$. But this implies that $\gcd(q_1,q_2,q_3,q_4)\neq 1$ which is a contradiction.

Thus, by \eqref{eqSP},
$\mult(A,\myC(p_1))=\mult(A,\myC(W_i))=\mathfrak{q} \, m= \deg(\cQ) \, \mult(A,\myC(V_i))$.
\end{proof}

\para

We finish this section stating the relationship between the transversality of $\cS$ and $\cP$ under the assumption that $\cP(\cS^{-1})$ does not have base points.

\para

\begin{theorem}\label{theorem-transversalSiffP}   Let $\cP$ and $\cQ$ be two birational projective parametrizations of the same surface $\myS$ such that $\cQ(\cS)=\cP$ and $\myB(\cQ)=\emptyset$. Then, $\cS$ is transversal if and only if $\cP$ is transversal.
\end{theorem}
\begin{proof}
Let $A\in \myB(\cS)=\myB(\cP)$. First we note that from Lemma \ref{L-mult}, and Corollary 5 in \cite{CoxPerezSendra2020}, it holds that
\[ \mult(A, \myC(W_1))^2= \mult(A, \myC(V_1))^2  \deg({Q})^2 =\mult(A, \myC(V_1))^2  \deg(\myS) \]
Using Corollary 9 in \cite{CoxPerezSendra2020}, we have that
\[ \mult(A, \myB(\cP))=\deg(\myS)\mult(A, \myB(\cS)). \]
Therefore,
\[ \mult(A,\myB(\cS))\,\mult(A,\myC(W_1))^2=\mult(A,\myB(\cP))\,\mult(A,\myC(V_1))^2. \]
Thus, $\cS$ is transversal if and only if $\cP$ is transversal.
 \end{proof}

\section{Proper Polynomial Reparametrization}\label{section-properPolRep}

\para

In this section, we deal with the central problem of the paper, namely, the determination, if they exist, of proper (i.e. birational) polynomial parametrizations of rational surfaces. For this purpose, we distinguish several subsections. In the first subsection, we fix the general assumptions and we propose our strategy. In the second subsection, we perform the theoretical analysis, and in the last subsection  we prove the existence of a linear subspace, computable from the input data,  and containing the solution to the problem.

\para

We start recalling what we mean with a polynomial projective parametrization.   We say that a projective parametrization is \textsf{polynomial} if its dehomogenization w.r.t. the fourth component, taking $t_i=1$ for some $i\in \{1,2,3\}$, is polynomial; note that the fourth component of a polynomial projective parametrization has to be a power of $t_i$ for some $i\in \{1,2,3\}$. Clearly, a similar reasoning is applicable w.r.t. other dehomogenizations. On the other hand, we say that a parametrization is \textsf{almost polynomial} if its fourth component is the power of a linear form.

\para

The important fact is that a rational surface admits a birational polynomial parametrization if and only if it admits a birational almost polynomial parametrization. Furthermore, if we have an almost polynomial parametrization, and its fourth component is a power of the linear form  $L_{3}^{*}(\ot)$, we may consider two additional linear forms  $L_{1}^{*},L_{2}^{*}$ such that $L^*=(L_{1}^{*}:L_{2}^{*}:L_{3}^{*})\in \myGS$ and then  the composition of the almost polynomial parametrization with $(L^*)^{-1}$ is a polynomial parametrization of the same surface.

\subsection{General assumptions and strategy}\label{subsec-general assumptions}

\para

In our analysis we have two main assumptions. Se \textbf{assume} that the rational surface $\myS$ admits a polynomial birational parametrization with empty base locus. Throughout the rest of the paper, let us \textbf{fix} one of these parametrizations  and denote it by $\cQ$; that is,
\begin{equation}\label{eq-q}
\cQ(\ot)=(q_1(\ot): q_2(\ot): q_3(\ot): q_4(\ot)),
\end{equation}
with  $q_i$  homogenous polynomials of the same degree such that $\gcd(q_1,\ldots,q_4)=1$,
is a proper polynomial parametrization of $\myS$ satisfying that $\myB(\cQ)=\emptyset$. Note  that, by Corollary 6 in \cite{CoxPerezSendra2020}, the degree of $\myS$ is then the square of a natural number. Moreover, we introduce a second assumption. We \textbf{assume} that we are given a transversal birational parametrization  of $\myS$. Throughout the rest of the paper, let us \textbf{fix}   $\cP$  as a transversal proper parametrization of $\myS$, and let $\cP$ be expressed as in  \eqref{eq-Param}.

\para

Our goal is to reach $\cQ$, or more precisely an almost polynomial parametrization of $\myS$, from $\cP$. For this purpose, first we observe that, since both $\cP$ and $\cQ$ are birational,
they are related by means of a birational map of $\proj2$, say $\cS_\cP$. More precisely, $\cS_\cP := \cQ^{-1} \circ \cP$. In the following, we represent  $\cS_\cP$ as
\begin{equation}\label{eq-Sp}
\cS_\cP(\ot) =(s_1(\ot):s_2(\ot):s_3(\ot)),
\end{equation}
where $\gcd(s_1,s_2,s_3)=1$. Note that, because of Theorem \ref{theorem-transversalSiffP}, since $\cP$ is transversal, then $\cS_\cP$ is transversal.
  In addition, let $\cR_{\cP}  :=\cS_{\cP}^{-1}(\ot)=\cP^{-1} \circ \cQ$. In the sequel, we represent $\cR_\cP$ as
\begin{equation}\label{eq-Rp}
\cR_{\cP}(\ot) =(r_1(\ot):r_2(\ot):r_3(\ot)),
\end{equation}
where $\gcd(r_1,r_2,r_3)=1$.

\para

So, in order to derive   $\cQ$ from $\cP$ it would be sufficient to determine $\cS_\cP$, and hence  $\cR_\cP$, because $\cQ=\cP(\cR_\cP)$. Furthermore  if, instead of determining $\cS_\cP$, we obtain
$\LcS_{\cP}:=L\circ \cS_\cP$, for some $L\in \myGS$, then instead of $\cQ$ we get
\[  \cP((\LcS_{\cP})^{-1})=\cP(\cR^{L^{-1}}_{\cP})=\cQ(L^{-1})=\cQ^{L^{-1}},\]
which is almost polynomial, and hence solves the problem. Taking into account this fact we make the following two considerations:
\begin{enumerate}
\item We can \textbf{assume} w.l.o.g. that $\myB(\cP)\neq \emptyset$. Indeed, if $\myB(\cP)=\emptyset$, by Theorem 10 and Corollary 9 in \cite{CoxPerezSendra2020}, we get that $\myB(\cS_\cP)=\emptyset$. Furthermore, by Corollary 7 in \cite{CoxPerezSendra2020}, we obtained that $\deg(\cS_\cP)=1$. Thus, using that $\cQ$ is indeed polynomial, we get that the fourth component of $\cP$ is the power of a linear form, and therefore the input parametrization $\cP$  would be already almost polynomial, and hence the problem would be solved.
\item We can \textbf{assume} w.l.o.g.  that $\cS_\cP$  satisfies whatever property  reachable by means of a left composition with elements in $\myGS$, as for instance those stated in Lemmas \ref{lemma-RatV1V2}, or \ref{lemma-rat-Si}, or \ref{lemma-singularlocus}. In particular, by Lemma \ref{lemma-transSSL}, the transversality is preserved. In other words, in the set $\myR$ of all birational transformations of $\proj2$, we consider  the equivalence relation $\thicksim$, defined as
$\cS \thicksim \cS^*$ if there exists $L\in \myGS$ such that $L\circ \cS=\cS^*$, and we work with the equivalence classes   in $\myR/\thicksim$. 
\end{enumerate}

\para

Therefore, our strategy will be to find a birational map $\cM$ of $\proj2$ such that $\cP(\cM^{-1})$ is almost polynomial. For this purpose, we will see that it is enough to determine  a dominant   rational transformation $\cM$ of $\proj2$ (later, we will prove that such a transformation is indeed birational)  such that
\begin{enumerate}
\item $\deg({\cM})=\deg(\cS_\cP)$.
\item $\myB({\cM})=\myB(\cS_\cP)$.
\item $\forall \, A\in \myB({\cM})$ it holds that
$\mult(A,\myB({\cM}))=\mult(A,\myB(\cS_\cP))$.
\end{enumerate}
 The difficulty is that both $ {\cM}$ and $\cS_\cP$  are unknown. Nevertheless, by Corollary 10 and Theorem 3 in \cite{CoxPerezSendra2020}, we have that
 \[ \deg(\cS_\cP)=\dfrac{\deg(\cP)}{\sqrt{\deg(\myS)}}. \]
 Note that $\deg(\cP)$ is given and $\deg(\myS)$ can be determined by applying, for instance, the formulas in \cite{JSC08} (see also \cite{PS-ISSAC}). On the other hand, taking into account Lemma \ref{lemma-BSBP}, we can achieve our goal by focusing on $\cP$. More precisely, we reformulate the above conditions into the equivalent following conditions.

 \para

\begin{conditions}\label{conditions} \
We say that a  rational dominant map ${\cM}$ of $\proj2$  satisfies Conditions \ref{conditions} if
\begin{enumerate}
\item $\deg(\cM)=\frac{\deg(\cP)}{\sqrt{\deg(\myS)}}$.
\item $\myB(\cM)=\myB(\cP)$,
\item $\mult(A,\myB(\cM))=\frac{\mult(A,\myB(\cP))}{\deg(\myS)}$ for all $A\in \myB(\cP)$.
\end{enumerate}
\end{conditions}

In the following subsections, we will see that   rational dominant maps  satisfying Conditions \ref{conditions} provide an answer to the polynomiality  problem.

\subsection{Theoretical analysis}

We start this analysis with some technical lemmas. For this purpose, $\myS, \cQ, \cP, \cS_\cP, \cR_\cP$   are as in the previous subsection. We recall that $\cQ(\cS_\cP)=\cP, \myB(\cQ)=\emptyset$, $\cR_\cP=\cS_{\cP}^{-1}$, $\cP$ is transversal, and hence $\cS_\cP$ is also transversal. Moreover, by Lemma \ref{lemma-BSBP}, $\cS_\cP$ satisfies Conditions \ref{conditions}.  Furthermore, in the sequel, let
\begin{equation}\label{eq-overlineS}
\overline{\cS}(\ot)=(\overline{s}_1(\ot):\overline{s}_2(\ot):\overline{s}_3(\ot)),
\end{equation}
with $\gcd(\overline{s}_1,\overline{s_2},\overline{s_3})=1$, be  dominant rational map of $\proj2$ satisfying Conditions \ref{conditions}.

\para

\begin{lemma}\label{lema-degSR}
Let $\cM$ be a birational map of $\proj2$. Then,
$\deg(\cM)=\deg(\cM^{-1})$.
\end{lemma}
\begin{proof} We use the notation introduced in Lemma \ref{lemma-RatV1V2}.
We take $L\in\myG$ such that
\begin{enumerate}
\item $\myC(V_{1}^{L})$  is rational (see  \eqref{eq-VL} for the definition of $V_{1}^{L}$ constructed from $\cM$, and Lemma \ref{lemma-RatV1V2} for the existence of $L$).
\item $\gcd(\eta_{1}^{L},\eta_{3}^{L})=1$, where $\cM^{-1}=(\eta_1: \eta_2: \eta_3)$.
\end{enumerate}
     In addition, we consider a projective transformation $N(\ot)$ in the parameters $\ot$ such that $\deg_{\ot}(V_{1}^{L}(\ox, N(\ot))=\deg_{t_2}(V_{1}^{L}(\ox, N(\ot))$ and
     $\deg_{\ot}(\eta_{1}^{L}(N(\ot))=\deg_{t_2}(\eta_{1}^{L}(N(\ot))=\deg_{t_2}(\eta_{3}^{L}(N(\ot))$. Then, it holds
\[
\begin{array}{cclr}
\deg(\cM) &= & \deg_{\ot}(V_{
1}^{L}(\ox,\ot) ) & \\
          &= & \deg_{\ot}(V_{1}^{L}(\ox,N(\ot)) & \text{($M$ is a proj. transf.)}\\
          & = & \deg_{t_2}(V_{1}^{L}(\ox, N(\ot)) & \text{(see above)} \\
          & = & \deg_{t_2}(\eta_{1}^{L}(N(x_1,h_1,x_3))/\eta_{3}^{L}(N(x_1,h_1,x_3)) & \text{(See Thm. 4.21 in \cite{PDSeSi}}) \\
          & = & \deg_{t_2}(\eta_{1}^{L}((N(x_1,h_1,x_3))) & \text{($\gcd(\eta_{1}^{L},\eta_{3}^{L})=1$)} \\
          & = & \deg(\cM^{-1}) &
\end{array}
\]
\end{proof}

\para

\begin{lemma}\label{Lemma-condition}
Let $\cM$ be a rational dominant map of $\proj2$.
If  $\cM$ satisfies Conditions \ref{conditions}, then $\cM$ is birational.
\end{lemma}
\begin{proof}
Since $\deg(\cM)=\deg(\cS_\cP)$, and $\mult(\myB(\cM))=\mult(\myB(\cS_\cP))$, by Theorem 7 (a) in \cite{CoxPerezSendra2020}, we have that $\mapdeg(\cM)=\mapdeg(\cS_\cP)$. So $\overline{\cS}$ is birational.
\end{proof}

\para

 Therefore, since we have assume above (see \eqref{eq-overlineS}) that $\overline{\cS}$ satisfies Conditions \ref{conditions}, $\overline{\cS}$ is birational. Let
\begin{equation}\label{eq-OverlineR}
\overline{\cR}(\ot)=\overline{\cS}^{-1}(\ot) =(\overline{r}_1(\ot): \overline{r}_2(\ot): \overline{r}_3(\ot))
\end{equation}
be its inverse. Clearly, $\overline{\cS}(\overline{\cR})=(t_1:t_2:t_3)$, which implies that $\overline{s}_i(\overline{\cR}(\ot))=t_i \, \wp(\ot)$, for $i\in \{1,2,3\}$, and where $\deg(\wp)=\deg(\overline{\cS})^2-1$ and hence $\deg(\wp)=\mult(\myB(\overline{\cS}))=\mult(\myB(\cS_\cP))=\mult(\myB(\cP))$. In the next result we prove that $\wp$ is directly related to $\myB(\overline{\cS})$, and using that $\myB(\overline{\cS})=\myB(\cS_\cP)$, we study the common factor appearing in the composition $\cS(\overline{\cR})$.  We start with a technical lemma.

\para

\begin{lemma}\label{lemma-composicion-derecha}
Let $L\in \myGS$. It holds that
\begin{enumerate}
\item $\myB(\cS_{\cP}^{L})=\myB(\cP^L)=\myB(\overline{\cS}^{\,L})$.
\item If $A\in \myB(\cS_{\cP}^L)$ then $\mult(A,\myB(\cS_{\cP}^L))=\mult(A,\cP^L)/\deg(\myS)$
\item $\overline{\cS}^{\,L}$   satisfies Conditions \ref{conditions}.
\item  $\cS_{\cP}^L$ is transversal
\item   If $\overline{\cS}$ is transversal,   then $\overline{\cS}^{L}$ is transversal.
\end{enumerate}
\end{lemma}
\begin{proof} \
(1) $A\in \myB(\cS_{\cP}^{L})$ iff $\cS_{\cP}^{L}(A)=\overline{0}$ iff $\cS_{\cP}(L(A))=\overline{0}$ iff $L(A)\in \myB(\cS_{\cP})=\myB(\cP)$ iff $\cP^L(A)=\cP(L(A))=\overline{0}$ iff $A\in \myB(\cP^L)$. Moreover, the second equality follows as in the previous reasoning, taking into account that $\overline{S}$ satisfies Condition \ref{conditions}, and hence $\myB(\cS_{\cP})=\myB(\overline{\cS})=\myB(\cP)$. \\
(2) follows taking into account that the multiplicity of a point on a curve, as well as the multiplicity of intersection, does not change under projective transformations. \\
(3) Condition (1) follows taking into account that $\deg(L)=1$. Statement (1) implies condition (2). For condition (3), we apply statement (2) and that $\mult(A,\myB(\overline{\cS}^{\, L}))=\mult(A,\myB(\cS))$ because the multiplicity of intersection does not change with $L$. \\
(4) and (5) follow  arguing as in  the proof of Lemma  \ref{lemma-basePbasePL} (4).
\end{proof}

\para

\begin{theorem}\label{T-constructionS} Let $\overline{\cS}$ be transversal.
 If $i\in \{1,2,3\}$ then $\overline{s}_i(\overline{\cR})=t_i\, \wp(\ot)$ where
  $\deg(\wp(\ot))=\mult(\myB(\cP))$  and such that $\wp$ is uniquely determined by $\myB(\cP)$.
\end{theorem}
\begin{proof}
We first observe that we can assume w.l.o.g. that no base point of $\cP$ is on the   line at infinity $x_3=0$. Indeed, let $L\in \myGS$ be such that $\myB(\cP)$ is contained in the affine plane $x_3=1$. We consider $\overline{\cS}^{\,*}:=\overline{\cS}^L=(\overline{s}_{1}^{\,*}:\overline{s}_{2}^{\,*}:\overline{s}_{3}^{\,*})$ and $\overline{\cR}^{\,*}:=(\overline{\cS}^{L})^{-1}$, then $\overline{s}_{i}^{\,*}(\overline{\cR}^{\,*})=\overline{s}_{i}^{\,*}(L^{-1}(\overline{\cR}))$, and $\overline{s}_{i}^{\,*}=\overline{s}_{i}(L)$; hence  $\overline{s}_{i}^{\,*}(\overline{\cR}^{\,*})=\overline{s}_{i}(\overline{\cR})$. In addition, because of Lemma \ref{lemma-composicion-derecha}, $\overline{\cS}^{\,*}$   satisfies the hypothesis of the theorem.

\smallskip

Let $\myC(\overline{V}_1)$ denote the curve associated to $\overline{\cS}$ as in \eqref{eq-V}.
By Lemma \ref{lemma-RatV1V2}, taking $L$ in the corresponding open subset of $\myGS$, we have that $\myC(\overline{V}_{1}^{L})$ is a rational curve. So, we assume w.l.o.g. that $\myC(\overline{V}_1)$ is rational. Let $\overline{\mathcal{V}}(\ox,h_1,h_2)$ be the rational parametrization of $\myC(\overline{V}_1)$ provided by Lemma \ref{lemma-RatV1V2}. We apply a M\"obious transformation $\phi\in \mathscr{G}(\mathbb{P}^{1}(\K))$ such that if $\overline{\cW}(\ox,h_1,h_2)=(\overline{w}_1(\ox,h_1,h_2):\overline{w}_2(\ox,h_1,h_2):\overline{w}_3(\ox,h_1,h_2)):=
\overline{\cV}(\ox,\phi(h_1,h_2))$ then  the affine parametrization $\overline{\rho}(\ox,h_1):=(\overline{w}_1(\ox,h_1,1)/\overline{w}_3(\ox,h_1,1), \overline{w}_2(\ox,h_1,1)/\overline{w}_3(\ox,h_1,1))$ is affinely surjective (see \cite{Sendra-Normal} and \cite{SWP}).

Now, let $A=(a_1:a_2:1)\in \myB(\cP)$. By Lemma \ref{lemma-BasePointS}, $P\in \myC(\overline{V}_{1})$. We observe that, by taking $L$ in the open subset of Lemma \ref{lemma-singularlocus}, we may assume that
\begin{equation}\label{eq-mult1} m_A:=\mult(A,\myC(\overline{V}_{1}))=\mult(A,\myC(\overline{s}_{i})), \,\,i\in \{1,2,3\}.
\end{equation}
We consider the polynomial
\[g_A=\gcd(\overline{w}_{1}(\ox,h_1,h_2)-a_1\overline{w}_{3}(\ox,h_1,h_2), \overline{w}_{2}(\ox,h_1,h_2)-a_2\overline{w}_{3}(\ox,h_1,h_2)). \]
 Since the affine parametrization has been taken surjective, we have that
\begin{equation}\label{eq-mult2}
 \deg_{\oh}(g_A)=m_A
\end{equation}
 and that for every root $t_0$ of $g_A$ it holds that $\overline{\rho}(t_0)=(a_1,a_2)$.
We write $\overline{w}_{i}$ as
\[ \overline{w}_i=g_A\cdot \overline{w}_{i}^{\,*}+a_i\overline{w}_{3},\,i=1,2. \]
On the other hand, we express $\overline{s}_i$ as
\[\overline{s}_i(\ot)=\overline{T}_{i,m_A}(\ot)t_3^{\deg(\overline{\cS})-m_A}+
\cdots+\overline{T}_{i,\deg(\overline{\cS})}(\ot),\,\,\,\]
where $\deg(\overline{T}_{i,j})=j,\,j\in\{m_A,\ldots,\deg(\overline{\cS})\},$ and $\overline{T}_{i,j}(\ot)=\sum_{k_1+k_2=j}(t_1-a_1t_3)^{k_1}(t_2-a_2 t_3)^{k_2}$. Therefore
\[
\overline{s}_i(\overline{\cW})=g_{A}^{m_A} \cdot
\left(\overline{T}_{i,m_A}(\overline{w}_{1}^{\,*}, \overline{w}_{2}^{\,*}) \, \overline{w}_{3}^{\,\deg(\overline{\cS})-m_A}+\cdots+g_{A}^{{\deg(\overline{\cS})-
m_A}}
\overline{T}_{i,\deg(\overline{\cS})}(\overline{w}_{1}^{\,*},
\overline{w}_{2}^{\,*})\right).
\]
In other words, $g_A$ divides $\overline{s}_i(\overline{\cW})$. Now, for $B=(b_1:b_2:1)\in \myB(\cP)$, with $A\neq B$, it holds that $\gcd(g_A,g_B)=1$, since otherwise there would exist a root $t_0$ of $\gcd(g_A,g_B)$, and this implies that $\overline{\rho}(t_0)=(a_1,a_2)=(b_1,b_2)=\overline{\rho}(t_0)$ which is a contradiction. Therefore, we have that
\begin{equation}\label{eq-factorization}
 \overline{s}_i(\overline{\cW})=\prod_{A\in \myB(\cP)}g_A(\ox,h_1,h_2)^{m_A}\, f_i(\ox,h_1,h_2)
\end{equation}
We observe that the factor defined by the base points does not depend on $i$. Thus, since $\overline{s}_i(\overline{\cW})$ does depend on $i$, we get that $f_i$ is the factor depending on $i$. Furthermore,
\[ \begin{array}{cclr}
\deg_{\oh}\left(\prod_{A\in \myB(\cP)}g_{A}^{m_A}\right) & =& \sum_{A\in \myB(\cP)}\deg_{\oh}(g_A)^{m_A} & \\
& = & \sum_{A\in \myB(\cP)} m_{A}^{2} &  \text{(see \eqref{eq-mult2})} \\
& = & \sum_{A\in \myB(\cP)} \mult(A,\myC(\overline{V}_1))^{2} &  \text{(see \eqref{eq-mult1})} \\
& = & \sum_{A\in \myB(\cP)} \mult(A,\myB(\overline{\cS}))   &  \text{($\overline{\cS}$ is transversal)} \\
& = & \sum_{A\in \myB(\cP)} \mult(A,\myB(\overline{\cP}))   &  \text{(See Conditions \ref{conditions})} \\
& = & \mult(\myB(\overline{\cP}))   &  \text{(See Def. \ref{def-multiplicity-BP})} \\
\end{array}
\]
Moreover, by Theorem 4.21 in  \cite{SWP}, since $\overline{\cW}$ is birational it holds that $\deg(\overline{\cW})=\deg(\myC(\overline{V}_1))=\deg(\overline{\cS})$. Hence, $\deg(\overline{s_i}(\overline{\cW}))=\deg(\overline{\cS})^2=\mult(\myB(\cP))+1$. Therefore, $f_i$ in \eqref{eq-factorization} is a linear form.

In this situation, let us introduce the notation $\ot^*:=(t_3,0,-t_1,t_1,t_2)$ and $\ot^{**}=(t_3,0,-t_1,\phi^{-1}(t_1,t_2))$.
Then, for $i\in \{1,2,3\}$, we have that
\[
\begin{array}{cclr}
t_i \,\wp &= & \overline{s}_i(\overline{\cR}) &  \\
& =& \overline{s}_{i}(\overline{\cV}(\ot^*)) & \text{(see Remark \ref{rem-V1V2})} \\
& =& \overline{s}_{i}(\overline{\cW}(\ot^{**})) & \text{(see definition of $\overline{\cW}$)}  \\
& =& \prod_{A\in \myB(\cP)}g_A(\ot^{**})^{m_A}\, f_i(\ot^{**})  & \text{(see \eqref{eq-factorization})}
\end{array}
\]
Taking into account that $\prod_{A\in \myB(\cP)}g_A(\ot^{**})^{m_A}$ does not depend on $i$, we get that $t_1\,f_2(\ot^{**}) =t_2\,f_1(\ot^{**}) $. This implies that $t_1$ divides $f_1(\ot^{**}) $, and since $f_1(\ot^{**})$ is linear we get that $t_1=\lambda f_1(\ot^{**})$ for $\lambda\in \K\setminus \{0\}$. Then, substituting above, we get $\lambda f_1(\ot^{**})f_{2}(\ot^{**})=t_2 f_1(\ot^{**})$, which implies that  $t_2=\lambda f_2(\ot^{**})$.
Similarly, for  $t_3=\lambda f_3(\ot^{**})$ .   Therefore, we get that
\[ \overline{s}_i(\overline{\cR})=t_i\, \lambda \prod_{A\in \myB(\cP)}g_A(\ot^{**})^{m_A}, \,\,\,\text{with $\lambda\in  \K\setminus \{0\}$} \]
This concludes the proof.
\end{proof}

\para

For the next theorem, we recall that $\cS_\cP=(s_1:s_2:s_3)$ with $\gcd(s_1,s_2,s_3)=1$; see \eqref{eq-Sp}.

\para

\begin{theorem}\label{T-constructionS2} Let   $\overline{\cS}$ be transversal.   If $i\in \{1,2,3\}$ then $ {s}_i(\overline{\cR})=Z_i(\ot)\,  \wp(\ot)$ where $Z_i$ is a linear form,
  $\deg(\wp(\ot))=\mult(\myB(\cP))$  and such that $\wp$ is uniquely determined by $\myB(\cP)$.
\end{theorem}
\begin{proof}
We first observe that we can assume w.l.o.g. that no base point of $\cP$ is on the   line at infinity $x_3=0$. Indeed, let $L\in \myGS$ be such that $\myB(\cP)$ is contained in the affine plane $x_3=1$. We consider ${\cS}^{\,*}:=\cS_{\cP}^L$$=({s}_{1}^{\,*}:{s}_{2}^{\,*}:{s}_{3}^{\,*})$,
$\overline{\cS}^{\,*}:=\overline{\cS}^L$, and  $\overline{\cR}^{\,*}:=(\overline{\cS}^{\,*})^{-1}$. Then
${s}_{i}^{\,*}(\overline{\cR}^{\,*})=s_{i}(L(L^{-1}(\overline{\cR})))=s_{i}(\overline{\cR})$.
In addition, because of Lemma \ref{lemma-composicion-derecha}, $\overline{\cS}^{\,*}$ and $\cS^*$ satisfy the hypotheses of the theorem.

\smallskip

Let $\myC(\overline{V}_1)$, $\overline{\mathcal{V}}(\ox,h_1,h_2)$, $\overline{\mathcal{W}}(\ox,h_1,h_2)=(\overline{w}_1:\overline{w}_2:\overline{w}_3)$ and $\overline{\rho}$ be as in the proof of Theorem \ref{T-constructionS}.

Now, let $A=(a_1:a_2:1)\in \myB(\cP)$. By Lemma \ref{lemma-BasePointS}, $P\in \myC(\overline{V}_{1})$. We recall that $\myB(\cP)=\myB(\overline{\cS})=\myB(\cS_\cP)$.
Let $\Omega_{3}^{\overline{\cS}}$ and $\Omega_{3}^{{\cS_\cP}}$ be the open subset in lemma \ref{lemma-singularlocus} applied to $\overline{\cS}$ and $\cS_\cP$, respectively. Taking $L\in \Omega_{3}^{\overline{\cS}}\cap \Omega_{3}^{\cS_\cP}$ (note that $\myGS$ is irreducible and hence the previous intersection is non-empty), we may assume that
\begin{equation}\label{eq-mult1b} m_A:=\mult(A,\myC(\overline{V}_{1}))=\mult(A,\myC(\overline{s}_{i})), \,\,i\in \{1,2,3\}.
\end{equation}
and
\begin{equation}\label{eq-multi1b-1}
\mult(A,\myC({V}_{1}))=\mult(A,\myC({s}_{i})), \,\,i\in \{1,2,3\}.
\end{equation}
Since $\cS_\cP$ and $\overline{\cS}$ are transversal, and taking into account Conditions \ref{conditions}, it holds that
\[ \mult(A,\myC(V_1))^2=\mult(A,\myB(\cS_\cP))=\mult(A,\myB(\overline{\cS}))=\mult(A,\myC(\overline{V}_1))^2 \]
Therefore,
\begin{equation}\label{eq-mult-equal}
 \mult(A,\myC(V_1))=m_A=\mult(A,\myC(\overline{V}_1)).
\end{equation}
We consider the polynomial $g_A=\gcd(\overline{w}_{1}-a_1\overline{w}_{3}, \overline{w}_{2}-a_2\overline{w}_{3}).$ Reasoning as in the Proof of Theorem \ref{T-constructionS} we get that
\begin{equation}\label{eq-mult2b}
 \deg_{\oh}(g_A)=m_A
\end{equation}
and that for every root $t_0$ of $g_A$ it holds that $\overline{\rho}(t_0)=(a_1,a_2)$.
We write $\overline{w}_{i}$ as $\overline{w}_i=g_A\cdot \overline{w}_{i}^{\,*}+a_i\overline{w}_{3}$ for $i=\in \{1,2\}.$

On the other hand, by  \eqref{eq-multi1b-1}, \eqref{eq-mult-equal}, we
have that $\mult(A,\myC(s_i))=m_A$. Therefore, we can express ${s}_i$ as
\[
{s}_i(\ot)=T_{i,m_A}(\ot)t_3^{\deg({\cS_\cP})-m_A}+
\cdots+T_{i,\deg({\cS_\cP})}(\ot),\,\,\,\]
where $\deg(T_{i,j})=j,\,j\in\{m_A,\ldots,\deg({\cS_\cP})\},$ and $T_{i,j}(\ot)=\sum_{k_1+k_2=j}(t_1-a_1t_3)^{k_1}(t_2-a_2 t_3)^{k_2}$. Therefore
\[
{s}_i(\overline{\cW})=g_{A}^{m_A} \cdot
\left(T_{i,m_A}(\overline{w}_{1}^{\,*}, \overline{w}_{2}^{\,*}) \, \overline{w}_{3}^{\,\deg( {\cS_\cP})-m_A}+\cdots+g_{A}^{{\deg({\cS_\cP})-
m_A}}
T_{i,\deg(\cS)}(\overline{w}_{1}^{\,*},
\overline{w}_{2}^{\,*})\right).
\]
In other words, $g_A$ divides ${s}_i(\overline{\cW})$. Now, for $B=(b_1:b_2:1)\in \myB(\cP)$, with $A\neq B$, reasoning as in the proof of Theorem \ref{T-constructionS}, it holds that $\gcd(g_A,g_B)=1$. Therefore, we have that
\begin{equation}\label{eq-factorization}
 {s}_i(\overline{\cW})=\prod_{A\in \myB(\cP)}g_A(\ox,h_1,h_2)^{m_A}\, f_i(\ox,h_1,h_2)
\end{equation}
Furthermore,
\[ \begin{array}{cclr}
\deg_{\oh}\left(\prod_{A\in \myB(\cP)}g_{A}^{m_A}\right) & =& \sum_{A\in \myB(\cP)}\deg_{\oh}(g_A)^{m_A} & \\
& = & \sum_{A\in \myB(\cP)} m_{A}^{2} &  \text{(see \eqref{eq-mult2b})} \\
& = & \sum_{A\in \myB(\cP)} \mult(A,\myC(\overline{V}_1))^{2} &  \text{(see \eqref{eq-mult1b})} \\
& = & \sum_{A\in \myB(\cP)} \mult(A,\myB(\overline{\cS}))   &  \text{($\overline{\cS}$ is transversal)} \\
& = & \sum_{A\in \myB(\cP)} \mult(A,\myB(\overline{\cP}))   &  \text{(See Conditions \ref{conditions})} \\
& = & \mult(\myB(\overline{\cP}))   &  \text{(See Def. \ref{def-multiplicity-BP})} \\
\end{array}
\]
Moreover, by Theorem 4.21 in \cite{SWP}, since $\overline{\cW}$ is birational it holds that $\deg(\overline{\cW})=\deg(\myC(\overline{V}_1))=\deg(\overline{\cS})$.
Hence, by Condition \ref{conditions}, $\deg({s_i}(\overline{\cW}))=\deg(\cS_\cP)\deg(\overline{\cS})=\deg(\cS_\cP)^2=\mult(\myB(\cP))+1$. Therefore, $f_i$ in \eqref{eq-factorization} is a linear form.

In this situation, let us introduce the notation $\ot^*:=(t_3,0,-t_1,t_1,t_2)$ and $\ot^{**}=(t_3,0,-t_1,\phi^{-1}(t_1,t_2))$.
Then, for $i\in \{1,2,3\}$, we have that
\[
\begin{array}{cclr}
 {s}_i(\overline{\cR})& =& {s}_{i}(\overline{\cV}(\ot^*)) & \text{(see Remark \ref{rem-V1V2})} \\
& =&  {s}_{i}(\overline{\cW}(\ot^{**})) & \text{(see definition of $\overline{\cW}$)} \\
& =& \prod_{A\in \myB(\cP)}g_A(\ot^{**})^{m_A}\, f_i(\ot^{**})  & \text{(see \eqref{eq-factorization})}
\end{array}
\]
This concludes the proof.
\end{proof}

\para

\begin{corollary}\label{C-constructionS2} If   $\overline{\cS}$ is transversal, there exists $L\in \myGS$ such that   $\overline{\cS}=\LcS_{\cP}$.
\end{corollary}
\begin{proof} From Theorem \ref{T-constructionS2}, we get that $\cS_{\cP}(\overline{\cR})=(Z_1(\ot): Z_2(\ot): Z_3(\ot))$,  where $Z_i$ is a linear form.
Thus, $\LcS_{\cP}=\overline{\cS},$ where $L\in \myGS$ is the inverse of $(Z_1, Z_2, Z_3)$.
\end{proof}

\para

\begin{corollary}\label{C-constructionS3}
  The following statements are equivalent
\begin{enumerate}
\item $\overline{\cS}$ is transversal.
\item There exists $L\in \myGS$ such that $\overline{\cS}=\LcS_{\cP}$.
\end{enumerate}
\end{corollary}
\begin{proof}
If (1) holds, then (2) follows from Corollary \ref{C-constructionS2}. Conversely, if (2) holds, then (1) follows from Lemma \ref{lemma-transSSL}
\end{proof}

\subsection{The Solution Space}

In this subsection we introduce a linear projective variety containing the solution to our problem and we show how to compute it. We start  identifying the set of all projective curves, including multiple component curves, of a fixed degree $d$, with the projective space (see \cite{Miranda},  \cite{SWP} or  \cite{Walker} for further details) $$\myV_d:=\mathbb{P}^{\frac{d (d+3)}{2}}(\K).$$
More precisely, we identify the projective curves of degree $d$ with the forms in $\K[\ot]$ of degree $d$, up to multiplication by non-zero $\K$-elements. Now, these forms are identified with the elements in $\myV_d$ corresponding to their coefficients, after fixing an order of the monomials. By abuse of notation, we will refer to the elements in $\myV_d$ by either their tuple of coefficients, or the associated form, or the corresponding curve.

\para

Let $\cM=(m_1(\ot):m_2(\ot):m_3(\ot))$, $\gcd(m_1,m_2,m_3)=1$, be a birational transformation of $\proj2$. We consider $\myV_{\deg(\cM)}$. Then, $m_1,m_2,m_3\in \myV_{\deg(\cM)}$. Moreover, in $\myV_{\deg(\cM)}$, we introduce  the projective linear subspace
\[ \cL(\cM):=\{ a_1 m_1(\ot)+ a_2 m_2(\ot) +a_3 m_3(\ot) \,|\, (a_1:a_2:a_3)\in \proj2\}. \]
We observe that if $\{m_1,m_2,m_3\}$ are linearly  dependent, then the image of $\proj2$ via $\cM^{-1}$ would be a line in $\proj2$ which is impossible because $\cM$ is birational on $\proj2$. Therefore, the following holds.

\para

\begin{lemma}\label{lemma-dim-L(S)}
If $\cM$ is a birational transformation of $\proj2$, then
$\dim(\cL(\cM))=2$.
\end{lemma}

\para

\noindent
Similarly, one has the next lemma

\para

\begin{lemma}\label{Lemma-LSLSL}
If $\cM$ is a birational transformation of $\proj2$ and $L\in \myGS$ then $\cL(\cM)=\cL({}^{L}\!\cM)$.
\end{lemma}

\para

\noindent
Furthermore, the following theorem holds

\para

\begin{theorem}\label{theorem-basesDeLS}
Let $\cM$ be a birational transformation of $\proj2$ and
let $\{n_1,n_2,n_3\}$ be  a basis of $\cL(\cM)$ and $\mathcal{N}:=(n_1:n_2:n_3)$. There exists $L\in \myGS$ such that ${}^{L}\!\cM=\mathcal{N}$.
\end{theorem}
\begin{proof}
Let $\cM=(m_1:m_2:m_3)$, with $\gcd(m_1,m_2,m_3)=1$. Since $m_1,m_2,m_3\in \cL(\cM)$, and $\{n_1,n_2,n_3\}$ is a basis of $\cL(\cM)$, there exist $(\lambda_{i,1}:\lambda_{i,2}:\lambda_{i,3})\in \proj2$ such that
\[ m_i=\sum \lambda_{i,j} n_j. \]
Since $\{m_1,m_2,m_3\}$ is also a basis of $\cL(\cM)$, one has that $L:=(\sum \lambda_{1,j} t_j:
\sum \lambda_{2,j} t_j: \sum \lambda_{3,j} t_j)\in \myGS)$ and $\mathcal{N}=L\circ \cM$.
\end{proof}

\para

\begin{remark}
Observe that, by Theorem \ref{theorem-basesDeLS}, all bases of $\cL(\cM)$ generate birational maps of $\proj2$.
\end{remark}

\para

\begin{corollary}\label{corollary-transversal}
Let $\cM$ be a birational transformation of $\proj2$. The following statements are equivalent
\begin{enumerate}
\item $\cM$ is transversal.
\item There exists a basis $\{n_1,n_2,n_3\}$   of $\cL(\cM)$ such that $(n_1:n_2:n_3)$ is transversal.
\item For all bases $\{n_1,n_2,n_3\}$   of $\cL(\cM)$ it holds that $(n_1:n_2:n_3)$ is transversal.
\end{enumerate}
\end{corollary}
\begin{proof}
It follows from Theorem \ref{theorem-basesDeLS} and Lemma \ref{lemma-transSSL}.
\end{proof}

\para

In the following results we analyze the bases of $\cL(\cS_\cP)$. So, $\cS,\cR,\cP, \cQ$ and $\overline{S}$ are as the in previous subsections.

\para

\begin{corollary}\label{corollary-condition}
Let $\{m_1,m_2,m_3\}$ a basis of $\cL(\cS_\cP)$. Then, $(m_1:m_2:m_3)$  satisfies Conditions \ref{conditions}.
\end{corollary}
\begin{proof}
It is a direct consequence of Theorem \ref{theorem-basesDeLS}.
\end{proof}

\para

\begin{corollary}\label{corollary-definitivo}
 If $\cM:=(m_1:m_2:m_3)$ is transversal and satisfies  Conditions \ref{conditions},   then $\{m_1,m_2,m_3\}$ is a basis of $\cL(\cS_\cP)$.
\end{corollary}
\begin{proof}
By Corollary \ref{C-constructionS3}, there exists $L\in \myGS$ such that $\cM=\LcS_{\cP}$. Now, by Lemma \ref{Lemma-LSLSL}, $\cL(\cS_\cP)=\cL(\cM)$. Taking into account that
$\{m_1,m_2,m_3\}$ are linearly independent, we get the result.
\end{proof}

\para

The previous results show that the solution to our problem lies in $\cL(\cS_\cP)$. However, knowing $\cL(\cS_\cP)$ implies knowing $\cS_\cP$, which is essentially our goal. In the following, we see how to achieve $\cL(\cS_\cP)$ by simply knowing $\myB(\cS_\cP)$ and the base point multiplicities; note that, under the hypotheses of this section, this information is given by  $\cP$.

\para

\begin{definition}\label{def-linear-system} Let $\cM$ be a birational transformation of $\proj2$.
We define the \textsf{linear system of base points of $\cM$}, and we denote it by $\myL(\cM)$, as the linear system of curves, of degree $\deg(\cM)$,
\[ \myL(\cM)=\{ f\in \myV_{\deg(\cM)}\,|\,   \mult(A,\myC(f))\geq \sqrt{\mult(A,\myB(\cM))}\,\,\forall A\in \myB(\cM) \} \]
Observe that $\myL(\cM)$ is the $\deg(\cM)$-linear system associated to the effective divisor
\[ \sum_{A\in \myB(\cM)} \sqrt{\mult(A,\myB(\cM))} \,\cdot A  \]
\end{definition}


\para

\begin{remark} We observe that if $\cM$ satisfies Conditions \ref{conditions}, in particular $\cS_\cP$, then $\myL(\cM)$ is the $\deg(\cS_\cP)$-degree linear system generated by the effective divisor
\[
 \sum_{A\in \myB(\cP)} \sqrt{\mult(A,\myB(\cP))} \,\cdot A \]
\end{remark}

\para

\noindent
The following lemma is a direct consequence  of the definition above.

\para

\begin{lemma}\label{Lemma-myLSLSL}
Let $\cM$ be a birational transformation of $\proj2$.
If $L\in \myGS$ then $\cL(\cM)=\cL({}^{L}\!\cM)$ and $\myL(\cM)=\myL({}^{L}\!\cM)$.
\end{lemma}

\para

Next lemma relates the $\cL(\cM)$ and $\myL(\cM)$.

\para

\begin{lemma}\label{Lemma-subset}
If $\cM$ is a transversal birational map of $\proj2$, then  $\cL(\cM) \subset \myL(\cM)$.
\end{lemma}
\begin{proof}
Let $\cM=(m_1:m_2:m_3)$, let $f\in \cL(\cM)$ and $A\in \myB(\cM)$. Then, $\deg(f)=\deg(\cM)$. On the other hand
\[ \begin{array}{cclr}
\mult(A,\myC(f)) & \geq & \min \{ \mult(A,\myC(m_i))\,|\, i\in \{1,2,3\}\} & \\
&=& \mult(A,\myC(V_{1})) & \text{(see Lemma \ref{lemma-BasePointS} (2))} \\
& =&\sqrt{\mult(A,\myB(\cM))}  & \text{($\cM$ is transversal).}    \end{array}
 \]
 Therefore, $f\in \myL(\cM)$.
 \end{proof}

\para

\begin{lemma}\label{lemma-dim2}
If $\cM$ is a transversal birational map of $\proj2$, then  $\dim(\myL(\cM))=2$.
\end{lemma}
\begin{proof} Let $\cM=(m_1:m_2:m_3)$. By Lemmas \ref{lemma-dim-L(S)} and \ref{Lemma-subset}, we have that $\dim(\myL(\cM))\geq 2$.
Let us assume that $\dim(\myL(\cM))=k>2$ and let $\{n_1,\ldots,n_{k+1}\}$ be a basis of $\myL(\cM)$ where $n_1=m_1,n_2=m_2,n_3=m_3$. Then $$\myL(\cM)=\{\lambda_1 n_1+\cdots +\lambda_{k+1} n_{k+1}\,|\, (\lambda_1:\cdots:\lambda_{k+1})\in \mathbb{P}^{k+1}(\K)\}.$$
Now, we take three points in $\proj2$ that will be crucial later. For their construction, we first
consider an open Zariski subset $\Sigma\subset \proj2$ where $\cM: \Sigma \rightarrow \cM(\Sigma)\subset \proj2$ is a bijective map. Then, $\cM(\Sigma)$ is a constructible set of $\proj2$ (see e.g. Theorem 3.16 in \cite{Harris:algebraic}). Thus, $\proj2 \setminus \cM(\Sigma)$ can only contain finitely many lines. On the other hand, we consider the open subset $\Omega_2\subset \myGS$ introduced in Lemma \ref{lemma-rat-Si}, and we take $L=(L_1:L_2:L_3)\in \Omega_2$ such that a non-empty open subset of $\myC(L_1)$ is included in  $\cM(\Sigma)$.
We take three points $B_1,B_2,B_3\in\Sigma$  (this, in particular, implies that $B_1,B_2,B_3\not\in \myB(\cM)$) such that:
\begin{enumerate}
\item $\cM(B_1)\neq \cM(B_2)$
\item  $\cM(B_1),\cM(B_2)\in   \myC(L_1)$,
\item $\cM(B_3)\not\in \myC(L_1)$; note that $\cM(B_1),\cM(B_2),\cM(B_3)$ are not on a line
\end{enumerate}
Since $\cM(B_1),\cM(B_2),\cM(B_3)$ are not collinear,  the system
\[ \left\{ \sum_{i=1}^{k+1} \lambda_{i} n_{i}(B_j)=0\right\}_{j\in\{1,2,3\}} \]
has solution. Let $(b_1:\cdots:b_{k+1})$ be a solution. Then, we consider the polynomials (say that $L_1:=a_1 t_1+a_2 t_2+ a_3 t_3$)
\[ f(\ot):=L_1(\cM)=a_1m_1+a_2m_2+a_3m_3,\,\,\, g(\ot):=b_1n_1+\cdots+b_{k+1} n_{k+1}. \]
 We have that $\myC(f)$ is irreducible because $L\in \Omega_2$. Moreover, $\deg(\myC(f))=\deg(\myC(g))$.
 In addition, $\myC(f)\neq \myC(g)$: indeed, $B_3\in \myC(g)$ and $B_3\not\in \myC(f)$ because otherwise
 \[ \left( \begin{array}{ccc} m_1(B_1) & m_2(B_1) & m_3(B_1) \\
 m_1(B_2) & m_2(B_2) & m_3(B_2) \\
 m_1(B_3) & m_2(B_3) & m_3(B_3) \\
 \end{array}\right) \left(\begin{array}{c} a_1 \\ a_2 \\ a_3 \end{array}\right)= \left(\begin{array}{c} 0 \\ 0 \\ 0 \end{array}\right),\]
 and since  $\cM(B_1),\cM(B_2),\cM(B_3)$ are not collinear we get that $a_1=a_2=a_3=0$ that is a contradiction.  Therefore, $\myC(f)$ and $\myC(g)$ do not share components. Thus, by B\'ezout's theorem the number of intersections of $\myC(f)$ and $\myC(g)$, properly counted, is $\deg(\cM)^2$.
 In addition, we oberve that $f\in \cL(\cM)\subset \myL(\cM)$ (see Lemma \ref{Lemma-subset}) and $g\in  \myL(\cM)$. Thus,
 \begin{equation}\label{eq-inclusion}
 \myB(\cM)\cup \{B_1,B_2\} \subset \myC(f)\cap \myC(g).
 \end{equation}
Therefore
\[
\begin{array}{cclr}
\deg(\cM)^2 & =& \displaystyle{\sum_{A\in \myC(f)\cap \myC(g)} \mult_{A}(\myC(f),\myC(g))} & \\
& \geq & \!\!\displaystyle{\sum_{A\in \myB(\cM)} \mult_{A}(\myC(f),\myC(g)) +\!\! \sum_{A\in \{B_1,B_2\}} \mult_{A}(\myC(f),\myC(g))} & \text{(see \eqref{eq-inclusion})}   \\
& \geq & \displaystyle{\sum_{A\in \myB(\cM)} \mult_{A}(\myC(f),\myC(g)) + 2} & \text{($B_1,B_2\not\in \myB(\cM)$)}   \\
& \geq & \displaystyle{\sum_{A\in \myB(\cM)} \mult(A,\myC(f))\,\mult(A,\myC(g)) + 2} &  \\
& \geq & \displaystyle{\sum_{A\in \myB(\cM)} \mult(A,\myB(\cM)) + 2} & \text{($f,g\in \myL(\cM))$)} \\
& = & \mult(\myB(\cM))+2 & \text{(see Def. \ref{definition-BS})} \\
&= & \deg(\cM)^2+1 & \text{(see Cor. 7 in \cite{CoxPerezSendra2020})}.
\end{array}
\]
which is a contradiction.
\end{proof}

\begin{theorem}\label{theorem-dim2myL(S)}
If $\cM$ is a transversal birational map of $\proj2$, then $\cL(\cM)=\myL(\cM)$.
\end{theorem}
\begin{proof}
By Lemma \ref{Lemma-subset}, $\cL(\cM)\subset \myL(\cM)$. Thus, using Lemmas \ref{lemma-dim-L(S)} and \ref{lemma-dim2}, we get the result.
\end{proof}

\section{Algorithm and Examples}\label{section-algorithms-examples}

\para

In this section, we use the previous results to derive an algorithm for determining polynomial parametrizations of rational surface, under the conditions stated in Subsection \ref{subsec-general assumptions}. For this purpose we first introduce an auxiliary algorithm for testing the transversality of parametrizations. In addition, we observe that we require to the input parametrization to be proper (i.e. birational). This can be checked for instance using the algorithms in \cite{PS-grado}.

\para

\begin{algorithm}[H] \caption{Transversality of a Parametrization}\label{AlgTrans}
\begin{algorithmic}[1]
\smallskip
\REQUIRE A rational proper projective parametrization ${\mathcal
P}(\ot)$
of an algebraic surface $\myS$.
\\
\STATE Compute  $\myB({\cP})=\bigcap_{i=1}^{4} \myC(p_i)$ and  $\mult(A, \myB({\cP}))=\mult_{A}(\myC(W_1),\myC(W_2))$ for every  $A\in \myB({\cP})$.

\IF {$\exists\, A\in \myB({\cP})$, such that $\mult(A, \myB({\cP}))\neq m_A^2$ for some  $m_A\in \Bbb N,\,m_A\geq 1$}
\STATE  {\bf return} ``\textsf{$\cP$ is not transversal''}.
\ENDIF
\IF {$\forall \, A\in \myB({\cP})$,  $\gcd(T_1, T_2, T_3, T_4) = 1$, where $T_i$ is the product of
the tangents, counted with multiplicities, to $\myC(p_i)$ at $A$,}
\STATE {\bf return} ``\textsf{$\cP$ is transversal}''   {\bf else return} ``\textsf{$\cP$ is not transversal}''.
\ENDIF
\end{algorithmic}
\end{algorithm}

\para

Observe that Step 2 in Algorithm \ref{AlgTrans} provides a first direct filter to detect some non-transversal parametrizations, and Step 5 applies the characterization in Lemma \ref{Lemma-traPgcd}. This justifies the next theorem

\para

\begin{theorem}
Algorithm \ref{AlgTrans} is correct.
\end{theorem}

\para

The following algorithm is the central algorithm of the paper.

\para

\begin{algorithm}[H] \caption{Birational Polynomial Reparametrization for Surfaces}\label{AlgPoli}
\begin{algorithmic}[1]
\smallskip
\REQUIRE A rational proper projective parametrization ${\mathcal
P}(\ot)=\left({p_{1}(\ot)}: {p_{2}(\ot)}:\,{p_{3}(\ot)}:\,{p_{4}(\ot)}\right)$
  of an algebraic surface $\myS$,   with $\gcd(p_1,\ldots,p_4)=1$.
  \IF{$p_4(\ot)$ is   of the form $(a_1t_1+a_2t_2+a_3t_3)^{\deg(\cP)}$,}
   \STATE
  Consider the projective transformation $L=(t_i,t_j,a_1t_1+a_2t_2+a_3t_3)$ where $i,j\in \{1,2,3\}$ are different indexes such if $\{k\}=\{1,2,3\}\setminus \{i,j\}$, then $a_k\neq 0$.
\STATE
  Compute the inverse $L^{-1}$ of $L$ and   {\bf return} ``$\cP(L^{-1})$  \textsf{is a  rational proper polynomial parametrization of $\myS$}''.
\ENDIF
\STATE Apply Algorithm  \ref{AlgTrans} to check whether $\cP$ is transversal. In the affirmative case go to the next Step. Otherwise   {\bf return} ``\textsf{Algorithm \ref{AlgPoli} is not applicable}''.
\STATE Compute $\deg(\myS)$ and $\deg({\cS})=\deg(\cP)/\sqrt{\deg(\myS)}$.


\STATE  Compute the $\deg({\cS})$-linear system associated to the effective divisor
\[\myL:=\sum_{A\in \myB({\cP})}   \sqrt{\mult(A, \myB({\cP}))/\deg(\myS)} \cdot A\]
\STATE Determine   $\overline{\cS}(\ot)=(\overline{s}_1(\ot): \overline{s}_2(\ot): \overline{s}_3(\ot))$, where $\{\overline{s}_1, \overline{s}_2, \overline{s}_3\}$ is a basis  of $\myL$.

\STATE Compute $\overline{\cR}(\ot)=\overline{\cS}^{-1}(\ot)$.

\STATE Compute $\cQ(\ot)=(q_1:q_2:q_3:q_4)$, where $\cQ(\ot)=\cP(\overline{\cR}(\ot))$.
\IF {$q_4(\ot)$ is   of the form $(a_1t_1+a_2t_2+a_3t_3)^{\deg(\cQ)}$,}
\STATE {\bf return} ``$\cQ(L^{-1})$ (where $L$ is as in Step 2) \textsf{is a  rational proper polynomial parametrization of $\myS$}'' \textbf{else return} ``\textsf{$\myS$ does not admit a polynomial proper parametrization with empty base locus}''.
\ENDIF
\end{algorithmic}
\end{algorithm}

\para

\begin{theorem}
Algorithm \ref{AlgPoli} is correct.
\end{theorem}
\begin{proof}
For the correctness of the first steps (1-4) we refer to the preamble in Section \ref{section-properPolRep} where the almost polynomial parametrizations are treated. For the rest of the steps,  we use the notation introduced in Section \ref{section-properPolRep} and we assume the hypotheses there, namely, $\cQ(\cS_{\cP})=\cP$ and $\myB(\cQ)=\emptyset$. Since $\cP$ is transversal, by Theorem \ref{theorem-transversalSiffP}, we have that $\S_\cP$ is transversal. Now, by Theorem \ref{theorem-dim2myL(S)}, $\myL=\cL(\cS_\cP)$. In this situation, by Theorem \ref{theorem-basesDeLS}, $\overline{S}=\LcS_{\cP}$ for some $L\in \myGS$. Therefore, $\cP(\overline{\cR})$ has to be almost polynomial, and hence the last step generates a polynomial parametrization. If the fourth component of $\cQ$, namely $q_4$ is not the power of a linear form, then   $\myB(\cQ)\neq \emptyset$.
\end{proof}

\para

\begin{remark} Let us comment some consequences and  computational aspects involved in the execution of the previous algorithms.
\begin{enumerate}
\item We observe that if Algorithm \ref{AlgPoli} returns a parametrization, then it is polynomial and its base locus is empty.
    \item In order to check the properness of $\cP$, one may apply, for instance,  the results in \cite{PS-grado} and, for determining   $\deg(\myS)$,   one may apply, for instance, the results in \cite{PS-ISSAC}, \cite{JSC08} or  \cite{PSV-Singularities}. For the computation of  $\overline{\cR}$ one may apply well known elimination techniques as resultants or Gr\"{o}bner basis; see e.g. \cite{schichoBir}.
\item In different steps of both algorithms one need to deal with the base points. Since the base locus is zero-dimensional, one may consider a decomposition of its elements in families of conjugate points, so that all further step can be performed exactly by introducing algebraic extensions of the ground field. For further details on how to deal with conjugate points we refer to Section 3.3. in \cite{SWP}
\end{enumerate}
\end{remark}

\para

\para

We finish this section illustrating the algorithm with some examples. The first two examples provide polynomial parametrizations, while in the third the algorithm detects that the input parametrization, although proper, is not transversal.

\para

\begin{example}\label{E-algorithm1} Let $\cP(\ot)=(p_1(\ot): p_2(\ot): p_3(\ot): p_4(\ot))$ be a rational parametrization of an algebraic surface $\myS$, where
\[
\begin{array}{ccl}
p_1&=& -6  t_3^4  t_1  t_2+6  t_3^2  t_2^2  t_1^2- t_3  t_2  t_1^4-2  t_3  t_2^3  t_1^2+5  t_3^3  t_1^2  t_2+3  t_3^3  t_1  t_2^2- t_2^6+3  t_3^2  t_1^4+3  t_3^2  t_2^4- \\
\noalign{\vspace*{1mm}}
&&  t_3  t_2^5-3  t_1^4  t_2^2-3  t_1^2  t_2^4+ t_3^3  t_1^3-6  t_3^4  t_1^2+3  t_3^5  t_1+3  t_3^3  t_2^3-6  t_3^4  t_2^2+2  t_3^5  t_2- t_1^6. \\
\noalign{\vspace*{2mm}}
p_2 &=& -(t_1- t_3)  t_3 (t_2^2+ t_1^2- t_1  t_3) (t_2^2+ t_1^2-2  t_3^2+ t_1  t_3).\\
\noalign{\vspace*{2mm}}
p_3 &=&  t_3^2  t_2^3  t_1+ t_3^2  t_1^3  t_2-3  t_3^4  t_1  t_2+39  t_3^2  t_2^2  t_1^2-8  t_3  t_2^2  t_1^3-4  t_3  t_1  t_2^4-4  t_3  t_2  t_1^4-8  t_3  t_2^3  t_1^2 +8  t_3^3  t_1^2  t_2
 \\
 \noalign{\vspace*{1mm}}
  &&+6  t_3^3  t_1  t_2^2+6  t_3^6-5  t_2^6-4  t_3  t_1^5+20  t_3^2  t_1^4+19  t_3^2  t_2^4-4  t_3  t_2^5-15  t_1^4  t_2^2-15  t_1^2  t_2^4+8  t_3^3  t_1^3\\
  \noalign{\vspace*{1mm}}&& -29  t_3^4  t_1^2+4  t_3^5  t_1+7  t_3^3  t_2^3-22  t_3^4  t_2^2-2  t_3^5  t_2-5  t_1^6. \\
\noalign{\vspace*{1mm}}
p_4  &=& (t_1^2+ t_2^2- t_3^2)^3.
\end{array}
\]
Applying the results in \cite{PS-grado}, one deduces that $\cP$ is proper.
We apply  Algorithm \ref{AlgPoli} in order to compute a  rational proper polynomial parametrization ${\cal Q}(\ot)$ of $\myS$, without base points, if it exists. Clearly, $\cP$ is not almost polynomial and hence steps 1-4 does not apply. In Step 5, we perform   Algorithm \ref{AlgTrans}. The base locus is (we denote by $\pm \imath$ the square roots of $-1$) $$ \myB({\cP})=\bigcap_{i=1}^{4} \myC(p_i)=\{(1:0:1),\,(1:\imath:0),\,(1:-\imath:0)\}.$$ Moreover, it holds that   $$\mult(A, \myB({\cP}))=\mult_{A}(\myC(W_1),\myC(W_2))=9, \,\,\forall A\in \myB({\cP}.)$$ Therefore, for every  $A\in \myB({\cP})$ we have that $\mult(A, \myB({\cP}))=m_A^2$ for some  $m_A\in \Bbb N,\,m_A\geq 1$. Thus, the necessary condition in Algorithm \ref{AlgTrans} is fulfilled. In addition, one may also check that the gcd of the tangents is 1, for each base point. As a consequence, we deduce that $\cP$ is transversal.

In Step 6 of Algorithm \ref{AlgPoli},   we get that   $\deg(\myS)=9$ (see  \cite{PSV-Singularities}).
Now, using that
$$\deg({\cS})=\deg(\cP)/\sqrt{\deg(\myS)}=6/3=2,$$ and that $$\mult(A, \myB({\cS}))=\mult(A, \myB({\cP}))/\deg(\myS)=9/9=1\,\quad  \text{for every  $A\in \myB({\cP})$},$$
we compute the $2$-linear system associated to the effective divisor
\[\sum_{A\in \myB({\cP})}    A.\]
For this purpose, one considers a generic polynomial of degree $2$ with undetermined coefficients (note that we have $6$ undetermined coefficients). We impose the three conditions, i.e.  $\{(1:0:1),\,(1:\pm \imath:0) \}$ should be simple points, and we get
\[\myL:=\lambda_1(-9t_1^2-9t_2^2+9t_1t_3+t_2t_3)+\lambda_2(-10 t_1^2-10 t_2^2+9 t_1t_3+t_3^2)+\lambda_3(t_1^2+t_2^2-t_3^2).\]
Let  \[\overline{\cS}(\ot)=(\overline{s}_1 : \overline{s}_2 : \overline{s}_3 )= (-9t_1^2-9t_2^2+9t_1t_3+t_2t_3: -10 t_1^2-10 t_2^2+9 t_1t_3+t_3^2: t_1^2+t_2^2-t_3^2),\] where $\{\overline{s}_1, \overline{s}_2, \overline{s}_3\}$ is a basis of $\myL$
Next, we compute $$\overline{\cR}(\ot)=\overline{\cS}^{-1}(\ot)=(\overline{r}_1(\ot): \overline{r}_2(\ot): \overline{r}_3(\ot))=$$
where
\[\begin{array}{ccl}
\overline{r}_1&=&81  t_1^2-162  t_1  t_2-162  t_1  t_3+71  t_3^2+151  t_2  t_3+80  t_2^2,
\\
\overline{r}_2&=& -9 (2 t_2+11  t_3) (t_1- t_2- t_3), \\
\overline{r}_3&=&181  t_3^2+82  t_2^2+81  t_1^2-162  t_1  t_2+182  t_2  t_3-162  t_1  t_3.
\end{array}\]
In the last step,  the algorithm returns
$$\cQ(\ot)=\cP(\overline{\cR}(\ot))=(t_1^3+ t_2  t_3^2- t_1  t_3^2- t_3^3: t_2(t_2- t_3)(t_2+ t_3):  t_2^3+ t_1  t_2^2+ t_3  t_2  t_1-4  t_1  t_3^2-5  t_3^3:  t_3^3)$$
 that is a  rational proper polynomial parametrization of $\myS$ with empty base locus. Note that the affine polynomial parametrization is given as
$$(t_1^3+ t_2- t_1-1,\, t_2(t_2- 1)(t_2+ 1),  t_2^3+ t_1  t_2^2+ t_2  t_1-4  t_1-5).$$
Observe that in this example we have introduced  $\pm \imath$. Nevertheless we could have considered conjugate points. More precisely, the base locus decomposes as
\[ \{(1:0:1)\} \cup \{(1:s:0)\,|\, s^2+1=0\} \]
Then, all remaining computations could have been carried out working in the field extension $\mathbb{Q}(\alpha)$ where $\alpha^2+1=0$.
\end{example}

\para

\begin{example}\label{E-algorithm2} Let $\cP(\ot)=(p_1(\ot): p_2(\ot): p_3(\ot): p_4(\ot))$ be a rational parametrization of an algebraic surface $\myS$, where\\
\[
\begin{array}{ccl}
p_1& = & \frac{2891876933101
}{7056}   t_2^4  t_3^2- \frac{94253497}{42}   t_2^5  t_3+
\frac{79182089}{24}   t_1^5  t_3- \frac{15185833}{35}   t_2^5  t_1+
\frac{230745016769}{19600}   t_2^4  t_1^2
 \\
\noalign{\vspace*{1mm}} &&
-
\frac{789948757}{280}   t_1^4  t_2^2-
\frac{314171}{4}   t_1^5  t_2-
\frac{3324893202046}{2205}   t_2^3  t_3^2  t_1+
\frac{17297334852139}{29400}   t_2^2  t_3^2  t_1^2\\
\noalign{\vspace*{1mm}} &&+ \frac{835536822991}{5880}   t_2^4  t_3  t_1
- \frac{3567593339657}{14700}   t_2^3  t_3  t_1^2+ \frac{8869391921}{420}   t_1^4  t_2  t_3+ \frac{199437407}{140}   t_1^3  t_2^3\\
\noalign{\vspace*{1mm}} &&+ \frac{56021820649}{144}   t_1^4  t_3^2+ \frac{925548000997}{630}   t_1^3  t_2  t_3^2
\\
\noalign{\vspace*{1mm}} &&- \frac{35094007283}{210}   t_1^3  t_2^2  t_3+ \frac{28561}{4}   t_2^6+ \frac{3455881}{16}   t_1^6,
 \\
\noalign{\vspace*{3mm}}
p_2 &=&  - \frac{1097019300247}{2352}   t_2^4  t_3^2- \frac{246980149}{56}   t_2^5  t_3
-\mbox{\scriptsize{$2485483$}}\,  t_1^5  t_3- \frac{32737835}{56}   t_2^5  t_1- \frac{35410335273}{3920}   t_2^4  t_1^2
\\
\noalign{\vspace*{1mm}} &&+ \frac{321945}{4}   t_1^4  t_2^2+ \frac{314171}{4}   t_1^5  t_2+ \frac{287134716635}{168}
   t_2^3  t_3^2  t_1- \frac{52659146973}{80}   t_2^2  t_3^2  t_1^2- \frac{35536353385}{294}   t_2^4  t_3  t_1\\
\noalign{\vspace*{1mm}} &&+ \frac{60928171523}{280}   t_2^3  t_3  t_1^2+ \frac{52899535}{3}   t_1^4  t_2  t_3- \frac{446331197}{140}   t_1^3  t_2^3- \frac{5296771655}{12}   t_1^4  t_3^2-
   \frac{49879553251}{30}   t_1^3  t_2  t_3^2\\
\noalign{\vspace*{1mm}} &&+ \frac{23802911463}{140}   t_1^3  t_2^2  t_3- \frac{257049}{16}   t_2^6,
\\
\noalign{\vspace*{3mm}}
 p_3&=& - \frac{2676488123101}{7056}   t_2^4  t_3^2+ \frac{94253497}{42}   t_2^5  t_3-
 \frac{379182089}{24}   t_1^5  t_3+
 \frac{15185833}{35}   t_2^5  t_1- \frac{219513256369}{19600}   t_2^4  t_1^2\\
\noalign{\vspace*{1mm}} &&+ \frac{797945837}{280}   t_1^4  t_2^2+ \frac{314171}{4}   t_1^5  t_2+ \frac{3079803152296}{2205}   t_2^3  t_3^2  t_1- \frac{16059945270739}{29400}   t_2^2  t_3^2  t_1^2\\
\noalign{\vspace*{1mm}} &&- \frac{786351504991}{5880}   t_2^4  t_3  t_1+ \frac{3371173762457}{14700}   t_2^3  t_3  t_1^2- \frac{9628594001}{420}   t_1^4  t_2  t_3- \frac{163616167}{140}   t_1^3  t_2^3\\
\noalign{\vspace*{1mm}} &&- \frac{51903261673}{144}   t_1^4  t_3^2- \frac{857765630677}{630}   t_1^3  t_2  t_3^2  + \frac{32679676343}{210}   t_1^3  t_2^2  t_3-
 \frac{28561}{4}   t_2^6- \frac{3455881}{16}   t_1^6, \\
\noalign{\vspace*{2mm}}
p_4&=&   (-5348  t_1^2  t_3+5525  t_2^2  t_3+169  t_1^2  t_2+757  t_1  t_2^2-10059  t_1  t_2  t_3)^2.
\end{array}
\]
Applying the results in \cite{PS-grado}, one deduces that $\cP$ is proper.
We apply  Algorithm \ref{AlgPoli}. Clearly, $\cP$ is not almost polynomial and hence steps 1-4 does not apply. In Step 5, we perform   Algorithm \ref{AlgTrans}. The base locus is  $$ \myB({\cP})=\{(0:0:1),(1:2:1), (5:7:1), (1/3:-1/7:1),(-13:7:1)\}.$$ Moreover, it holds that
   $$\mult(A, \myB({\cP})) =4$$ for every  $A\in \myB({\cP})$ except for $A=(0:0:1)$ that satisfies that $\mult(A, \myB({\cP}))=16$. Thus, the necessary condition in Algorithm \ref{AlgTrans} is fulfilled. In addition, one may also check that the gcd of the tangents is 1, for each base point. As a consequence, we deduce that $\cP$ is transversal.
Now, using that
$$\deg({\cS})=\deg(\cP)/\sqrt{\deg(\myS)}=6/2=3,$$ and that $$\mult(A, \myB({\cS}))=\mult(A, \myB({\cP}))/\deg(\myS)=1,$$
 for every  $A\in \myB({\cP})$ except for $A=(0:0:1)$ that satisfies that $\mult(A, \myB({\cS}))=4$,
we compute the $3$-linear system associated to the effective divisor
\[4\, (0:0:1)+(1:2:1)+ (5:7:1)+ (1/3:-1/7:1)+(-13:7:1)  .\]
We get that $\myL=\lambda_1 \overline{s}_1+\lambda_2 \overline{s}_2+\lambda_3 \overline{s}_3$ where
$$\begin{array}{ccl}
\overline{s}_1 &=&  \frac{203971}{12}  t_1^2 t_3- \frac{1463501}{84}  t_2^2 t_3+
\frac{3373732}{105}  t_1 t_2 t_3+ \frac{1859}{4}  t_1^3+ \frac{169}{2}  t_2^3-
\frac{169}{2}  t_1^2 t_2- \frac{438913}{140}  t_1 t_2^2,
\\
\noalign{\vspace*{2mm}}
\overline{s}_2 &=&
 \frac{37443}{2}  t_1^2 t_3- \frac{538707}{28}  t_2^2 t_3+  \frac{140997}{4}  t_1 t_2 t_3- \frac{507}{4}  t_2^3-507 t_1^2 t_2- \frac{71637}{28}  t_1 t_2^2,
  \\
\noalign{\vspace*{2mm}}
\overline{s}_2 &=& \frac{26747}{2}  t_1^2 t_3- \frac{384007}{28}  t_2^2 t_3-338 t_1^2 t_2- \frac{50441}{28}  t_1 t_2^2+ \frac{100761}{4}  t_1 t_2 t_3- \frac{507}{4}  t_2^3.
\end{array}$$
So, we take, for instance, $\overline{\cS}(\ot)=(\overline{s}_1(\ot): \overline{s}_2(\ot): \overline{s}_3(\ot))$ and we compute $\overline{\cR}(\ot)=\overline{\cS}^{-1}(\ot)=(\overline{r}_1(\ot): \overline{r}_2(\ot): \overline{r}_3(\ot)) $ where
$$\begin{array}{ccl}
\overline{r}_1 &=&  \dfrac{1}{11} (-34331  t_2+7140  t_1+39091  t_3)
 (-1240370879  t_2^2+4693319730  t_2  t_3
 \\
\noalign{\vspace*{2mm}}
&&-957816090  t_1  t_2-4096303731  t_3^2+26989200  t_1^2+1272637170  t_1  t_3),\\
\noalign{\vspace*{2mm}}
\overline{r}_2 &=&  -\dfrac{7}{3} (-5349  t_3+3821  t_2) (-1240370879  t_2^2+4693319730  t_2  t_3-957816090  t_1  t_2\\
\noalign{\vspace*{2mm}}
&&-4096303731  t_3^2+26989200  t_1^2+1272637170  t_1  t_3),
\\
\noalign{\vspace*{2mm}}
\overline{r}_3 &=&  9122349600  t_1^2  t_3-6081566400  t_1^2  t_2+5962839694227  t_3^3-13840668860013  t_2  t_3^2\\
\noalign{\vspace*{2mm}}
&&+10640657052993  t_2^2  t_3-2711599696487  t_2^3+503701536030  t_1  t_2^2\\
\noalign{\vspace*{2mm}}
&&-1409880894660  t_1  t_2  t_3+985048833510  t_1  t_3^2.
\end{array}
$$
\noindent
Finally,  we obtain
$$\cQ(\ot)=\cP(\overline{\cR}(\ot))=(t_1^2+t_2^2-t_2t_3: -t_1t_2-t_2^2+t_1t_3: -t_1^2+t_3^2-t_2t_3: (t_2-t_3)^2).$$
Since $q_4(\ot)=(t_2-t_3)^2$,   the algorithm returns
$$\cQ((t_1, t_2, t_2-t_3)^{-1})=(t_1^2+t_2t_3: -t_2^2-t_1 t_3: -t_1^2+t_3^2-t_2t_3: t_3^2)$$ that is a  rational proper polynomial parametrization of $\myS$ with empty base locus. Note that the affine polynomial parametrization is given as
$$(t_1^2+t_2, -t_2^2-t_1, -t_1^2+1-t_2).$$
\end{example}

\para

\begin{example}\label{E-algorithm3} Let $\cP(\ot)=(p_1(\ot): p_2(\ot): p_3(\ot): p_4(\ot))$ be a rational parametrization of an algebraic surface $\myS$, where

\[
\begin{array}{ccl}
p_1&=& (-14065142 t_1^3 t_3+29410550 t_2^3 t_3-29410550 t_2 t_1^2 t_3+14065142 t_2^2 t_1 t_3+27633480 t_1^4 \\
\noalign{\vspace*{2mm}}
&&-46976541 t_1 t_2^3+64760061 t_1^3 t_2)^2,
 \\
\noalign{\vspace*{3mm}}
p_2&=&
 15452942581758441/7 \, t_2 t_1^6 t_3- 317479084729363299/49 \  t_2^6 t_1 t_3
  \\
\noalign{\vspace*{2mm}}
&&- 68267697305871459/7 \, t_2^5 t_1^2 t_3- 18666824719928010/7 \, t_2^5 t_1 t_3^2
 \\
\noalign{\vspace*{2mm}}
&&+ 212684946864036627/49  t_2^4 t_1^2 t_3^2+ 37333649439856020/7  t_2^3 t_1^3 t_3^2 \\
\noalign{\vspace*{2mm}}
&&-2927680573060371 t_2^2 t_1^5 t_3- 18666824719928010/7  t_2 t_1^5 t_3^2 \\
\noalign{\vspace*{2mm}}
&&+ 10954535298967494/7  t_2^3 t_1^5+1789545850442280 t_2^2 t_1^6-1587369926524977 t_2 t_1^7 \\
\noalign{\vspace*{2mm}}
&&-700212410256675 t_1^6 t_3^2+1255537783884564 t_1^7 t_3+ 69932525820304176/7  t_2^7 t_1 \\
\noalign{\vspace*{2mm}}
&&- 202783295759585328/49  t_2^6 t_1^2-3042203660729001 t_2^5 t_1^3+ 48537853394156778/7  t_2^7 t_3 \\
\noalign{\vspace*{2mm}}
&&- 123497677483306851/49  t_2^6 t_3^2+ 399414081398977842/49  t_2^4 t_1^3 t_3 \\
\noalign{\vspace*{2mm}}
&&-4193865500723721 t_2^8-987075578994849 t_1^8-217339297920270 t_2^4 t_1^4 \\
\noalign{\vspace*{2mm}}
&&- 54876861278152701/49  t_2^2 t_1^4 t_3^2+ 4276901329956240/7  t_2^3 t_1^4 t_3, \\
\noalign{\vspace*{2mm}}
p_3& = &   3/7 (24511557 t_1^4-64760061 t_1^2 t_2^2+11755445 t_2 t_1^2 t_3+38554704 t_1 t_2^3-1125425 t_2^2 t_1 t_3 \\
\noalign{\vspace*{2mm}}
&&-11755445 t_2^3 t_3+1125425 t_1^3 t_3) (-151106809 t_2^4+97487778 t_2^3 t_3+269939512 t_1 t_2^3 \\
\noalign{\vspace*{2mm}}
&&+59811570 t_2^2 t_1 t_3-151106809 t_1^2 t_2^2-97487778 t_2 t_1^2 t_3+98258706 t_1^4-59811570 t_1^3 t_3) \\
\noalign{\vspace*{2mm}}
p_4&=&   (24511557 t_1^4-64760061 t_1^2 t_2^2+11755445 t_2 t_1^2 t_3+38554704 t_1 t_2^3-1125425 t_2^2 t_1 t_3 \\
\noalign{\vspace*{2mm}}
&&-11755445 t_2^3 t_3+1125425 t_1^3 t_3)^2.
\end{array}\]
Applying the results in \cite{PS-grado}, one gets that $\cP$ is proper. However, when applying
  Algorithm \ref{AlgTrans}, we get that $$\myB({\cP})= \{(0:0:1),(1:2:1), (5:7:1), (1/3:-1/7:1),(-13:7:1)\}$$ and  that $\mult(A, \myB({\cP}))=4$ for every  $A\in \myB({\cP})$ except for $A=(0:0:1)$ where $\mult(A, \myB({\cP}))=44$.  Since  $\mult(A, \myB({\cP}))=44$, which is not the square of a natural number, the algorithm returns that   $\cP$ is not transversal. Thus,   we can not apply Algorithm  \ref{AlgPoli}.
\end{example}


\section{Conclusions}
Some crucial difficulties in many applications, and algorithmic questions, dealing with surface parametrizations are, on one hand, the presence of base points and, on the other, the existence of non-constant denominators of the parametrizations. In this paper, we have seen how to provide a polynomial parametrization with empty base locus, and hence an algorithm to avoid the two complications mentioned above, if it is possible. For this purpose, we have had to introduce, and indeed impose, the notion of transversal base locus. This notion directly affects to the transversality of the tangents at the base points of the algebraic plane curves $V_i$ or $W_i$ (see \eqref{eq-V} and \eqref{eq-W}). This, somehow, implies that in general one may expect transversality in the input. In any case, we do deal here with the non-transversal case and we leave it as an open problem. We think that using the ideas, pointed out by J. Schicho in \cite{SchichoIssac}, on  blowing up the base locus, one might transform the given problem (via a finite sequence of Cremone transformations and projective transformations) into the case of transversality.

\section*{Acknowledgements}
This work has been partially supported by FEDER/Ministerio de Ciencia, Innovación y Universidades-Agencia Estatal de Investigación/MTM2017-88796-P (Symbolic Computation:
new challenges in Algebra and Geometry together with its applications). Authors belong to the Research Group ASYNACS (Ref. CT-CE2019/683).

\para


\end{document}